\documentclass[11pt]{article}

\usepackage{hyperref}
\usepackage[draft]{fixme}
\usepackage{mathrsfs}
\usepackage[a4paper]{geometry}
\usepackage{amssymb, amsthm, amsfonts, amsxtra, amsmath}
\usepackage{latexsym}
\usepackage{mathabx}
\usepackage{tensor}
\usepackage{pdfpages}
\usepackage{epic}
\usepackage{fouridx}
\usepackage{parskip} 
\makeatletter 
\def\thm@space@setup{%
  \thm@preskip=\parskip \thm@postskip=0pt
}
\makeatother
\usepackage{enumerate}
\usepackage{bbm}

\DeclareMathOperator{\id}{id}

\DeclareMathOperator{\ctau}{\tau}

\DeclareMathOperator{\dyn}{\mathrm{dyn}}
\DeclareMathOperator{\Spec}{\mathrm{Spec}}

\newcommand{\Circt}{{\mathop{\ooalign{$\ovoid$\cr\hidewidth\raise-.05ex\hbox{$\scriptstyle\mathsf T\mkern3.5mu$}\cr}}}} 
\newcommand{\Circtv}[1]{\underset{#1}{\mathop{\ooalign{$\ovoid$\cr\hidewidth\raise-.05ex\hbox{$\scriptstyle\mathsf T\mkern3.5mu$}\cr}}}} 
\newcommand{\smCirct}{\mathop{\ooalign{$\scriptstyle\ovoid$\cr\hidewidth\raise-.05ex\hbox{$\scriptscriptstyle\mathsf T\mkern2.8mu$}\cr}}}  

\newcommand{\su}{\mathfrak{su}}

\newcommand{\C}{\mathbb{C}}
\newcommand{\R}{\mathbb{R}}
\newcommand{\Z}{\mathbb{Z}}
\newcommand{\N}{\mathbb{N}}

\newcommand{\Hsp}{\mathcal{H}}

\newcommand{\Mor}{\mathrm{Mor}}
\newcommand{\alg}{\mathrm{alg}}

\newcommand{\itimes}{\underset{I}{\otimes}}

\newcommand{\GrDA}[3]{{}_{#2}#1_{#3}} 
\newcommand{\Grt}[3]{#1{\tiny {\begin{pmatrix} #2\\#3\end{pmatrix}}}} 

\newcommand{\Unit}{\mathbf{1}}
\newcommand{\UnitC}[2]{\Grt{\mathbf{1}}{#1}{#2}}

\newcommand{\Gr}[5]{\,\tensor*[^{#2}_{#4}]{#1}{^{#3}_{#5}}}
\newcommand{\Gru}[3]{\Gr{#1}{}{}{#2}{#3}}
\newcommand{\Grd}[3]{\Gr{#1}{}{}{#2}{#3}}

\newcommand{\aotimes}{\underset{\alg}{\otimes}}

\newtheorem{Theorem}{Theorem}[section]
\newtheorem*{Theo*}{Theorem}
\newtheorem{Lem}[Theorem]{Lemma}
\newtheorem{Prop}[Theorem]{Proposition}
\newtheorem{Cor}[Theorem]{Corollary}

\theoremstyle{definition}
\newtheorem{Def}[Theorem]{Definition}
\newtheorem{Rem}[Theorem]{Remark}
\newtheorem{Exa}[Theorem]{Example}

\newtheorem{Rems}[Theorem]{Remarks}

\numberwithin{equation}{section}

\begin{document}
\title{C$^*$-algebraic partial compact quantum groups}

\author{Kenny De Commer\thanks{Department of Mathematics, Vrije Universiteit Brussel, VUB, B-1050 Brussels, Belgium, email: {\tt kenny.de.commer@vub.ac.be}} \thanks{Supported by the FWO grant G.0251.15N. This work is also part of the project supported by the NCN-grant 2012/06/M/ST1/00169.}}

\date{}

\maketitle

\begin{abstract}
\noindent In this paper, we introduce C$^*$-algebraic partial compact quantum groups, which are quantizations of topological groupoids with discrete object set and compact morphism spaces. These C$^*$-algebraic partial compact quantum groups are generalisations of Hayashi's compact face algebras to the case where the object set can be infinite. They form the C$^*$-algebraic counterpart of an algebraic theory of partial compact quantum groups developed in an earlier paper by the author and T. Timmermann, the correspondence between which will be dealt with in a separate paper. As an interesting example to illustrate the theory, we show how the dynamical quantum $SU(2)$ group, as studied by Etingof-Varchenko and Koelink-Rosengren, fits into this framework. 
\end{abstract}




\section*{Introduction}

The concept of a \emph{compact quantum group of face type} was introduced by T. Hayashi \cite{Hay1}. A compact quantum group of face type can be interpreted as (a function algebra on a) compact quantum groupoid with a classical, finite object set, but with the source and target maps of the quantum arrow space `delocalized', that is, the corresponding embeddings of the function algebra on the object set are not central. Closely related to these are Ocneanu's \emph{double triangle algebras} \cite{Oc1,PZ1}, and \emph{weak Hopf C$^*$-algebras} \cite{BNS1}, where the object set is no longer assumed classical (see \cite{Sch1,Sch2} for a detailed discussion of the correspondence).  

In this paper, we will introduce the notion of \emph{C$^*$-algebraic partial compact quantum group}, which is a generalisation of Hayashi's construction to the case where the object set can be infinite (but is still discrete). Contrary to the approach in \cite{Hay1}, which follows \cite{DK94}, our definition is more in the spirit of Woronowicz's definition of a compact quantum group \cite{Wor98}, but contains a non-trivial extra density condition. The precise connection with Hayashi's work, as well as with the algebraic theory of partial compact quantum groups developed in \cite{DCT2}, will be dealt with in a separate paper \cite{DCT1}. 

We also make a partial version of Woronowicz's \emph{compact matrix pseudogroups} \cite{W1}. Indeed, although their definition is more restrictive than the one for compact quantum groups, compact matrix pseudogroups and their partial generalisations seem more suitable for constructing examples. 

As an illustration of the theory, we will show how the dynamical quantum $SU(2)$ group \cite{EtV1,KoR1} can be interpreted as a C$^*$-algebraic partial compact quantum group. In fact, this example fits in a much more general framework which will be considered in \cite{DCT1}. We however treat the case of dynamical quantum $SU(2)$ in isolation here, as it is more amenable to a direct, hands-on approach. For example, we completely determine the representation theory of the function algebra on the dynamical quantum $SU(2)$ group. We also mention that a different approach to the operator algebraic implementation of dynamical quantum $SU(2)$ is treated in \cite{Tim1}. The dynamical quantum $SU(2)$ that we treat will be at a parameter $q>0$. The `root of unity case' was considered in \cite[Section 13]{Hay2}, see also \cite{CT1} and \cite{EN1}.

Let us come to the precise contents of this paper. 

The first four sections deal with the general theory. In the \emph{first section}, we introduce C$^*$-algebraic partial compact quantum \emph{semi}groups and their representations. In the \emph{second section}, we impose certain density conditions on these structures to arrive at the notion of a \emph{C$^*$-algebraic partial compact quantum group}. The main result of this section is the existence of an invariant integral in this case. In the \emph{third section}, we introduce \emph{C$^*$-algebraic partial compact matrix pseudogroups}, and prove that they give rise to C$^*$-algebraic partial compact quantum groups. Then, in the \emph{fourth section}, we show how, by a universal construction, C$^*$-algebraic partial compact matrix pseudogroups can be created from purely algebraic data. 

In the final sections, we treat the example of dynamical quantum $SU(2)$. In the \emph{fifth section}, we introduce the dynamical quantum $SU(2)$ group, and show how it fits in the theory of C$^*$-algebraic partial compact quantum groups. In the final \emph{sixth section}, we classify the representations of the function algebra on dynamical quantum $SU(2)$.

\emph{Acknowledgements}: This paper originated from closely related joint work with T. Timmermann, whom I would warmly like to thank for all discussions and pertinent questions. Thanks are also due to E. Koelink, who introduced the author to the dynamical quantum $SU(2)$ group, to L. Va$\breve{\textrm{\i}}$nerman for discussions on quantum groupoids, and to M. Yamashita.

\section{C$^*$-algebraic partial compact quantum semigroups}

For $A$ a C$^*$-algebra, we denote by $M(A)$ the multiplier C$^*$-algebra of $A$. All tensor products $\otimes$ of C$^*$-algebras in this paper will be minimal. We denote by $[\,\cdot\,]$ the closed linear span of a subset of a Banach space. The C$^*$-algebra of all bounded operators on a Hilbert space $\Hsp$ is written $B(\Hsp)$, while the compact operators are denoted $B_0(\Hsp)$. 

\subsection{C$^*$-algebraic partial compact quantum semigroups}

\begin{Def}\label{DefCpcqsg} Let $I$ be a set, in the following referred to as the \emph{object set}. We call \emph{C$^*$-algebraic $I$-partial compact quantum semigroup} $\mathscr{G}$ a triple consisting of 
\begin{itemize}
\item a (not necessarily unital) C$^*$-algebra $A$,
\item a family of orthogonal self-adjoint projections $\UnitC{k}{l}\in A$ for $k,l\in I$, and
\item  a (not necessarily unital)
$^*$-homomorphism \[\Delta: A\rightarrow M(A\otimes A),\] 
\end{itemize}
satisfying the following conditions:
\begin{enumerate}[(a)]
\item[(U1)] $\underset{k,l\in I}{\sum} \UnitC{k}{l}$ converges strictly to the unit in $M(A)$,
\item[(U2)] $\Delta\left(\UnitC{k}{l}\right) = \underset{m\in I}{\sum}\UnitC{k}{m}\otimes \UnitC{m}{l}$ strictly for all $k,l\in I$,
\item[(C)] $\Delta$ is coassociative: \[(\Delta\otimes \id)\circ \Delta = (\id\otimes \Delta)\circ \Delta.\]
\end{enumerate}
\end{Def} 

\begin{Rems} 
\begin{enumerate}
\item $A$ is to be interpreted as the `function algebra $C_0(\mathscr{G})$' on $\mathscr{G}$, see Example \ref{ExaPSG}.
\item We allow the possibility that a $\UnitC{k}{l}$ is zero.
\item In (C), we interpret $(\Delta\otimes \id)$ and $(\id\otimes \Delta)$ as the unique \emph{strictly continuous} extensions to $M(A\otimes A)$. 
\end{enumerate}
\end{Rems} 

\begin{Lem} With $\Delta(1) = \sum_{k,l,m} \UnitC{k}{m}\otimes \UnitC{m}{l}$, we have \begin{equation}\label{CondDiWeak}(A\otimes A)\Delta(1) = [(A\otimes A)\Delta(A)].\end{equation} 
\end{Lem}

Note that the sum defining $\Delta(1)$ is strictly converging to a self-adjoint projection in $M(A\otimes A)$.

\begin{proof} Let us first show that $\supseteq$ holds. As $\Delta$ is a $^*$-homomorphism and $\Delta(A)\subseteq M(A\otimes A)$, it is by (U1) sufficient to show that $\Delta\left(\UnitC{k}{l}\right) \in (A\otimes A)\Delta(1)$. This is immediate by (U2).

Let us show now that $\subseteq$ holds. Again by (U1), it is sufficient to show that $(a\otimes b)\Delta(1) \in  (A\otimes A)\Delta(A)$ for $a$ in some $A\UnitC{k}{m}$ and $b$ in some $A\UnitC{l}{n}$. But since then $a= a\UnitC{k}{m}$ and $b= b\UnitC{l}{n}$, \[(a\otimes b)\Delta(1) = (a\otimes b)\Delta\left(\UnitC{k}{n}\right) \in (A\otimes A)\Delta(A).\]
\end{proof}

\begin{Rem} By the property \eqref{CondDiWeak}, $\Delta$ extends uniquely to a $^*$-homomorphism \[\Delta: M(A)\rightarrow M(A\otimes A)\] with value in the unit precisely $\Delta(1)$. 
\end{Rem}

\begin{Exa}\label{ExaPSG} We will call \emph{partial semigroup} a structure satisfying the axioms for a category, \emph{except} possibly for the existence of units.\footnote{That is, we have an object set $I$ and morphisms spaces $\Mor(k,l)$ between objects, together with associative multiplications $\Mor(k,l)\times \Mor(l,m)\rightarrow \Mor(k,m)$ sending $(g,h)$ to $gh$. Note that we use here multiplication in stead of composition, that is, we interpret $h\circ g = gh$.} Let us say that a partial semigroup is \emph{topological} if the arrow space $\mathscr{G}$ and the object space $I$ are topological (Hausdorff) spaces, and all structure morphisms are continuous maps. Let us call \emph{partial compact semigroup} (over $I$) a topological partial semigroup whose object set $I$ is \emph{discrete}, and which is \emph{proper} in the sense that the map \[\mathscr{G}\rightarrow I\times I, \quad g\mapsto (s(g),t(g)),\] assigning to an arrow its source and target object, is a proper map. Because of the discreteness of $I$, this simply means that all arrow spaces $\mathscr{G}(k,l)$ are compact.  

Consider now $C_0(\mathscr{G})$, so that $C_b(\mathscr{G}\times \mathscr{G}) \cong M(C_0(\mathscr{G})\otimes C_0(\mathscr{G}))$. We claim that $C_0(\mathscr{G})$, together with the orthogonal projections $\UnitC{k}{l}(g) = \delta_{g\in \mathscr{G}(k,l)}$ and the $^*$-homomorphism \[\Delta: C_0(\mathscr{G}) \rightarrow M(C_0(\mathscr{G})\otimes C_0(\mathscr{G})),\quad \Delta(f)(g,h) = \left\{\begin{array}{llll} f(gh)&\textrm{if} &g,h \textrm{ multipliable},\\ 0& \textrm{if}& g,h\textrm{ not multipliable},\end{array}\right.\] which is easily seen to be well-defined by continuity of the structure maps, is a C$^*$-algebraic $I$-partial compact quantum semigroup.  Indeed, the conditions in Definition \ref{DefCpcqsg} are immediately checked.

 It is not difficult to see, by Gelfand duality, that any C$^*$-algebraic $I$-partial compact quantum semigroup $(A,\Delta,\{\UnitC{k}{l}\})$, with $A$ commutative, is of this form. Indeed, denote $\mathscr{G} = \Spec(A)$. By (U1), the $\UnitC{k}{l}$ provide a decomposition $\mathscr{G} = \sqcup_{k,l\in I} \mathscr{G}(k,l)$ with each $\mathscr{G}(k,l)$ a compact (possibly empty) space. Denote \[\mathscr{G}^{(2)} = \bigsqcup_{k,l,m\in I}\left( \mathscr{G}(k,l)\times \mathscr{G}(l,m)\right) \subseteq \mathscr{G}\times \mathscr{G}.\] Then \[C_0\left(\mathscr{G}^{(2)}\right) \cong \Delta(1)(C_0(\mathscr{G})\otimes C_0(\mathscr{G})),\] and $\Delta$ dualizes to a continuous map $\mathscr{G}^{(2)} \rightarrow \mathscr{G}$. Condition (U2) gives that this product restricts to multiplications $\mathscr{G}(k,l)\times \mathscr{G}(l,m)\rightarrow \mathscr{G}(k,m)$, and condition (C) shows this product is associative.
 \end{Exa}

\subsection{Representations of C$^*$-algebraic partial compact quantum semigroups}\label{SectionRepSemi}

Let $I$ be a set. For
$\Hsp=\underset{k,l}{\bigoplus} \,\Grd{\Hsp}{k}{l}$ an $I$-bigraded Hilbert space, we denote by
$p_{kl}^{\Hsp} \in \mathcal{B}(\Hsp)$ the projections onto the homogeneous components. We further write
\begin{align*}
  \lambda^{\Hsp}_{k} &= \sum_{l} p_{kl}^{\Hsp}, &
  \rho^{\Hsp}_{l} &= \sum_{k} p_{kl}^{\Hsp},
\end{align*}
 the sums converging in the strong operator topology.

We will in the following definition use the \emph{leg numbering notation}, e.g. $X_{23} = 1\otimes X$ or $1\otimes X\otimes 1$ etc.

 \begin{Def} \label{def:corepresentation} Let $\mathscr{G} = (A,\Delta,\{\UnitC{k}{l}\})$ be a
   $C^{*}$-algebraic $I$-partial compact quantum semigroup. 

   A \emph{representation} of $\mathscr{G}$ is given by an $I$-bigraded Hilbert
   space $\Hsp$ together with an element $X \in M(A \otimes B_0(\Hsp))$
   satisfying the following conditions.
   \begin{enumerate}[a)]
   \item[(Co1)] With $\Delta \otimes \id$ extended to the multiplier algebra, we have
     \begin{align} \label{eq:corep}
     (\Delta \otimes \id)(X) = X_{13}X_{23}.  
   \end{align}
 \item[(Co2)] For all $k,l,m,n\in I$, \[(\UnitC{k}{m}\otimes 1)X(\UnitC{l}{n}\otimes 1) = (1\otimes p_{kl}^{\Hsp})X(1\otimes p_{mn}^{\Hsp}).\]
     
   \end{enumerate}
  A representation is called \emph{row- and column finite dimensional}, briefly
   \emph{rcfd}, if $\lambda^{\Hsp}_{k}$ and $\rho^{\Hsp}_{m}$ have finite rank for all $k,m\in I$.
\end{Def}

\begin{Rem} 
The rcfd condition is the proper generalization of finite dimensionality to the partial case. Indeed, the condition that $\Hsp$ itself be finite dimensional is in general too strong. 
\end{Rem} 

\begin{Exa} \label{exa:corep-trivial}
 Denote by $e_{km}$ the matrix units in $B(l^{2}(I))$ with respect to the standard basis $\{\delta_k\}$. Consider $l^2(I)$ with the diagonal $I^2$-grading $\GrDA{l^2(I)}{k}{l}= \delta_{k,l} \C\delta_k$.
 Then the sum
 \begin{align*}
  E = \sum_{k,m} \UnitC{k}{m} \otimes e_{km} 
\end{align*}
converges strictly in $M(A\otimes B_0(l^{2}(I)))$ to an rcfd representation. We call $E$ the \emph{trivial representation} of $\mathscr{G}$. 
\end{Exa}

\begin{Exa}\label{ExaClasRep} Let $\mathscr{G}$ be an $I$-partial compact semigroup as in Example \ref{ExaPSG}. Then an element $X \in M(C_0(\mathscr{G}) \otimes B_0(\Hsp))$ corresponds to a $\sigma$-strong$^*$-continuous uniformly bounded map \[\pi: \mathscr{G}\rightarrow B(\Hsp).\] If $X$ satisfies (Co2), we obtain that $g\in \mathscr{G}(k,m)$ satisfies $\pi(g) \in B(\GrDA{\Hsp}{m}{m},\GrDA{\Hsp}{k}{k})$, and $\pi(g)$ zero on all other components. Condition (Co1) gives that $\pi(gh) = \pi(g)\pi(h)$ when $g,h$ are multipliable. 

If we also ask that $X$ is \emph{non-degenerate}, that is, $X(C_0(\mathscr{G})\otimes B_0(\Hsp)) = C_0(\mathscr{G})\otimes B_0(\Hsp)$, then we see that $\GrDA{\Hsp}{k}{l}=0$ for $k\neq l$, i.e.\ the bigrading on $\Hsp$ is just a grading by $I$. The reader can easily verify that all uniformly bounded non-degenerate continuous representations of $\mathscr{G}$ (defined in the obvious manner) arise in this way. 
\end{Exa}

We will need to know how to tensor representations of a C$^*$-algebraic $I$-partial compact quantum semigroup $\mathscr{G}= (A,\Delta,\{\UnitC{k}{l}\})$. Let $X$ and $Y$ be $\mathscr{G}$-representations on respective $I$-bigraded Hilbert spaces $\Hsp$ and $\mathcal{K}$. Define $\Hsp \itimes\mathcal{K}$ as the Hilbert space
\begin{align*}
      \Hsp \itimes \mathcal{K} := \bigoplus_{k,l,m} \Grd{\Hsp}{k}{l} \otimes \Grd{\mathcal{K}}{l}{m}\subseteq \Hsp\otimes \mathcal{K}.
    \end{align*}
Then $\Hsp\itimes \mathcal{K}$ is again $I$-bigraded, the bigradation being given in the above decomposition by the $k$ and $m$-indices. 

Let $P$ be the orthogonal projection of $\Hsp\otimes \mathcal{K}$ onto $\Hsp\itimes \mathcal{K}$. Then it is easily checked, by (Co2), that $X_{12}Y_{13}P_{23} = X_{12}Y_{13}$. Hence we can interpret 
    \begin{align*}
      X \Circt Y:= X_{12}Y_{13}\in M(A \otimes B_0(\Hsp \itimes \mathcal{K})).
    \end{align*}
 It is then easily seen to be a representation of $\mathscr{G}$ on the Hilbert space $\Hsp\itimes \mathcal{K}$. Indeed, clearly
\[(\Delta \otimes \id)(X_{12}Y_{13}) = X_{13}X_{23}Y_{14}Y_{24} =
    X_{13}Y_{14}X_{23}Y_{24}.\] Also (Co2) is immediately verified. 
    
If $X$ and $Y$ are rcfd representations, then so is $X\Circt Y$.

\section{C$^*$-algebraic partial compact quantum groups}

\begin{Def}\label{DefCpcqg} Let $\mathscr{G}= (A,\Delta,\{\UnitC{k}{l})$ be a C$^*$-algebraic $I$-partial compact quantum semigroup. We call $\mathscr{G}$ a \emph{C$^*$-algebraic $I$-partial compact quantum group} if also the following conditions are satisfied:
\begin{enumerate}[(a)]
\item[(U3)] $\UnitC{k}{k}\neq 0$ for all $k\in I$. 
\item[(D1)]  We have \begin{equation}\label{CondDi}(A\otimes A)\Delta(1) = [(A\otimes 1)\Delta(A)] = [(1\otimes A)\Delta(A)].\end{equation} 
\item[(D2)] With $P=\sum_{k} \UnitC{k}{k}\in M(A)$, and $A_P = PAP$, we have \begin{equation}\label{CondDii} A=[(\omega\otimes \id)\Delta(A_P)\mid \omega \in A^*] = [(\id\otimes \omega)\Delta(A_P)\mid \omega \in A^*].\end{equation}
\end{enumerate}
\end{Def} 

\begin{Rems}
\begin{enumerate}
\item
The condition (U3) is a non-degeneracy condition: if $\UnitC{k}{k}$ were zero, $k$ could simply be dropped from the set $I$ since then also $\UnitC{k}{l}=0$ and $\UnitC{l}{k}=0$ for all $l$ using (U2) and the upcoming Lemma \ref{LemEqRel}.
\item Note that the inclusions of the right hand sides in the left hand side are automatically true in condition (D1), by (U1) and (U2).
\item Further comments on condition (D2) will be given in Remark \ref{RemEquiv}. Note that in the case of a compact quantum group ($|I|= 1$), condition (D2) follows immediately from condition (D1). 
\end{enumerate}
\end{Rems}

\begin{Exa} Let $\mathscr{G}$ be an \emph{$I$-partial compact group}, that is, an $I$-partial compact semigroup (see Example \ref{ExaPSG}) whose underlying categorical structure is a groupoid (and, in particular, a genuine category with identity maps). We claim that $(C_0(\mathscr{G}),\Delta,\{\UnitC{k}{l}\})$ is a C$^*$-algebraic $I$-partial compact quantum group. 

Indeed, the existence of identity arrows in a groupoid ensures that (U3) is satisfied. The density conditions of (D1) follow from the injectivity and properness of the map \[\mathscr{G}^{(2)}\rightarrow \mathscr{G}\times \mathscr{G},\quad (g,h)\mapsto (g,gh).\] The first density condition (D2) follows since if $\mathscr{G}(l,k)\neq \emptyset$, then also $\mathscr{G}(k,l) \neq \emptyset$, so for any chosen $g\in \mathscr{G}(k,l)$ we have a homeomorphism \[\mathscr{G}(l,k)\rightarrow \mathscr{G}(k,k),\quad h\mapsto gh,\] whence \[C(\mathscr{G}(l,k)) = \{(\mathrm{ev}_g\otimes \id)\Delta(f)\mid f\in C(\mathscr{G}(k,k))\}.\] The other density condition in (D2)
 follows similarly.
 
Conversely, if $(A,\Delta,\{\UnitC{k}{l}\})$ is a C$^*$-algebraic $I$-partial compact quantum group with $A$ commutative, then the associated $I$-partial compact semigroup $\mathscr{G}= \Spec(A)$ is an $I$-partial compact group. Indeed, condition (D1) gives that multiplication is left and right cancellative, and by (U3) the $\mathscr{G}(k,k)$ are non-empty. Hence the $\mathscr{G}(k,k)$ are compact groups, see e.g. \cite[Proposition 3.2]{MVD1}. 
 
 To conclude that $\mathscr{G}$ is a groupoid, it suffices to show that $\mathscr{G}(k,l)\neq \emptyset$ implies $\mathscr{G}(l,k) \neq \emptyset$. But this follows from (D2) by the upcoming Lemma \ref{LemEqRel}. 
 \end{Exa}
 
\begin{Rem}\label{RemRaum} The discussion at the end of the previous example also shows that the condition (D2) is necessary. Indeed, if $A$ is the function algebra on the category with two objects $\{1,2\}$, two identity arrows and one arrow $1\rightarrow 2$, then $(A,\Delta,\{\UnitC{k}{l}\})$, with $\Delta$ dual to composition and all $\UnitC{k}{l}$ as before except $\UnitC{2}{1}=0$, satisfies all requirements except (D2) (we thank S. Raum for this observation).
\end{Rem} 

\begin{Exa}
Let $\mathscr{G}$ be a (discrete) groupoid with object set $I$. Let  $C^*_u(\mathscr{G})$ be the universal groupoid C$^*$-algebra of $\mathscr{G}$, that is, $C^*_u(\mathscr{G})$ is generated by elements $\theta_g$ for $g$ an arrow in $\mathscr{G}$, with the relations \[\theta_g \theta_h = \left\{\begin{array}{llll} \theta_{gh}&\textrm{if} &g,h \textrm{ multipliable},\\ 0& \textrm{if}& g,h\textrm{ not multipliable},\end{array}\right.\] and with $\theta_g^* = \theta_{g^{-1}}$.

Then $(C^*_u(\mathscr{G}),\Delta,\{\UnitC{k}{l}\})$ is a C$^*$-algebraic $I$-partial compact quantum group by means of the orthogonal projections $\UnitC{k}{l} = \delta_{k,l} \theta_{\id_k}$ and the coproduct \[\Delta(\theta_g) = \theta_g\otimes \theta_g.\]
Indeed, it is easily seen that the above defines an $I$-partial compact quantum semigroup. As C$^*_u(\mathscr{G})$ admits a regular representation on $l^2(\mathscr{G})$, also (U3) is satisfied. Then (D1) follows from the surjectivity of the maps \[\mathscr{G}^{(2)}\rightarrow \{(g,h)\mid t(g) = t(h)\},\quad (g,h) \mapsto (gh,h),\]\[\mathscr{G}^{(2)}\rightarrow \{(g,h)\mid s(g) = s(h)\},\quad (g,h) \mapsto (g,gh),\] while (D2) is in this case a direct consequence of (D1) since the projection P of condition (D2) equals 1.
\end{Exa}

Let us return now to general C$^*$-algebraic $I$-partial compact quantum groups. The following lemma is elementary but important.

\begin{Lem}\label{LemEqRel}
Let $I$ be a set and let $(A,\Delta,\{\UnitC{k}{l}\})$  be a $C^{*}$-algebraic $I$-partial compact quantum group. Then $k\sim l \Leftrightarrow \UnitC{k}{l}\neq 0$ is an equivalence relation on $I$.
\end{Lem}

\begin{proof}
The relation is reflexive by (U3) and transitive by (U2). If $\UnitC{k}{l}\neq 0$, then by (D2) it must be contained in $[(\omega \otimes \id)(\Delta(\UnitC{l}{l}A\UnitC{l}{l})) | \omega \in A^{*}]$, whence $\UnitC{l}{k}A\UnitC{l}{k}\neq 0$ and $\UnitC{l}{k} \neq 0$, so $\sim$ is symmetric.
\end{proof}

\begin{Rem} It is not clear if the condition (D2) is \emph{equivalent} with $\sim$ being an equivalence relation. At least in the commutative and cocommutative case this holds, as we have shown above. 
\end{Rem} 

We next introduce the notion of invariant integral for a C$^*$-algebraic partial compact quantum group. Let us first introduce some more notation. 

\begin{Def}\label{Defstproj} Let $\mathscr{G} = (A,\Delta,\{\UnitC{k}{l}\})$ be a C$^*$-algebraic $I$-partial compact quantum semigroup. We write \[\lambda_k = \sum_{m}\UnitC{k}{m},\qquad \rho_m = \sum_{k}\UnitC{k}{m},\] which give well-defined projections in $M(A)$. We call them respectively \emph{source} and \emph{target} projections (corresponding to the respective objects $k$ and $m$). 
\end{Def}

\begin{Def} Let $\mathscr{G}=(A,\Delta,\{\UnitC{k}{l}\})$ be a C$^*$-algebraic $I$-partial compact quantum semigroup. An \emph{invariant integral} for $\mathscr{G}$ consists of a weight $\phi: A^+ \rightarrow [0,+\infty]$ satisfying the following conditions.
\begin{enumerate}[a)]
\item[(I1)] For all $k,m$ with $\UnitC{k}{m}\neq 0$, \[\phi(\UnitC{k}{m}) = 1.\]
\item[(I2)] For all $a\in A^+$, \[\phi(a) = \sum_{k,m\in I} \phi\left(\UnitC{k}{m}a\UnitC{k}{m}\right).\]
\item[(I3)] For all $a\in A^+$ and all states $\omega\in A^*$, \begin{equation}\label{EqInvL} \phi((\omega \otimes \id)\Delta(a)) = \sum_{k} \omega(\lambda_k)\phi(\lambda_ka\lambda_k),\end{equation} \begin{equation}\label{EqInvR} \phi((\id\otimes \omega)\Delta(a)) = \sum_{m}\omega(\rho_m)\phi(\rho_ma\rho_m),\end{equation} with the convention $0\cdot \infty = 0$. 
\end{enumerate} 
\end{Def}

Clearly, the formula \[\phi_{km}(a) = \phi(\UnitC{k}{m}a\UnitC{k}{m})\] defines a (bounded) positive functional $\phi_{km}$ on $A$. If $\UnitC{k}{m}\neq 0$ it is a state, otherwise it is the zero functional. By abuse of language, we will in the following refer to the complete family of $\phi_{km}$ as `states', so the reader should bear in mind that some of them can be zero functionals.

It is also clear by (I2) that $\phi$ is completely determined by the family
$\{\phi_{km}\}$.  

In terms of the $\phi_{km}$, the left and right invariance properties
\eqref{EqInvL} and \eqref{EqInvR} take the following form.  

\begin{Lem}
For all $a\in A$ and all $k,l\in I$,
\begin{align} \label{eq:invariance}  (\id \otimes
\phi_{kl})(\Delta(a)) &= \sum_{m}  \phi_{ml}(a)\UnitC{m}{k} , & (\phi_{kl} \otimes \id)(\Delta(a))&= \sum_{m} \phi_{km}(a) \UnitC{l}{m}  ,
\end{align}
where the sums converge strictly.
\end{Lem}
\begin{proof} For $\omega\in A^*$ positive and $a\in A^+$, we have \begin{eqnarray*} \omega((\id\otimes \phi_{kl})(\Delta(a))) &=& \omega(\rho_k (\id\otimes \phi_{kl})(\Delta(\rho_l a\rho_l))\rho_k) \\ &=& \phi_{kl}((\omega(\rho_k\,\cdot\,\rho_k)\otimes \id)\Delta(\rho_l a\rho_l))\\ &=& \phi((\omega(\rho_k\,\cdot\,\rho_k)\otimes \id)\Delta(\rho_l a\rho_l)) \\ &=& \sum_m \omega(\UnitC{m}{k})\phi\left(\UnitC{m}{l}a\UnitC{m}{l}\right).\end{eqnarray*} This implies the first equality in \eqref{eq:invariance}. The second equality follows similarly. 
\end{proof} 

Conversely, if \eqref{eq:invariance} holds for a family of states on $A$ with supports on the $\UnitC{k}{l}A\UnitC{k}{l}$, then it is clear that their sum will define an invariant weight. 

The relations \eqref{eq:invariance} can also  be rewritten in terms of the associative
convolution product on $A^*$  defined by \[(\chi*\omega)(a) =
(\chi\otimes \omega)\Delta(a).\] Let us write \begin{equation}\label{DefSpB} \Gr{B}{k}{m}{l}{n} = \{\omega \in A^* \mid \forall a\in A, \omega(a) = \omega\left(\UnitC{k}{m}a\UnitC{l}{n}\right)\}.\end{equation}Then the convolution product restricts to products \[\Gr{B}{k}{m}{l}{n}\times \Gr{B}{m}{p}{n}{q}\rightarrow \Gr{B}{k}{p}{l}{q},\] all other products between homogeneous components being zero. The left and right invariance properties \eqref{EqInvL} and \eqref{EqInvR} can now be written in terms of the $\phi_{km}\in \Gr{B}{k}{m}{k}{m}$ as \begin{equation}\label{EqInvLp}\omega*\phi_{km} = \omega(\UnitC{p}{k})\phi_{pm},\qquad \forall \omega \in \Gr{B}{p}{k}{q}{k},\end{equation}
\begin{equation}\label{EqInvRp}\phi_{km}*\omega = \omega(\UnitC{m}{q})\phi_{kq},\qquad \forall \omega \in \Gr{B}{m}{p}{m}{q}.\end{equation}

\begin{Theorem}\label{TheoInvInt} Each C$^*$-algebraic $I$-partial compact quantum group admits a unique invariant integral.
\end{Theorem} 

We will split the proof of Theorem \ref{TheoInvInt} into several steps, setting the stage so that eventually the arguments of \cite{MVD1} can be applied almost verbatim. We fix in the following a C$^*$-algebraic $I$-partial compact quantum group $\mathscr{G} = (A,\Delta,\{\UnitC{k}{l}\})$.

We will refer to families of states satisfying \eqref{EqInvL} (or equivalently \eqref{EqInvLp}) as a \emph{left invariant integral}, and to those satisfying \eqref{EqInvRp} as \emph{right invariant integral}.

\begin{Lem}\label{LemLeftisRight} Let $\{\phi_{km}\}$ be a left, and $\{\psi_{km}\}$ a right invariant integral for $\mathscr{G}$. Then $\phi_{km}= \psi_{km}$ for all $k,m$. 
\end{Lem} 
\begin{proof} By the invariance properties, and the fact that $\UnitC{k}{k}\neq 0$, we have \[\phi_{km}  = \psi_{kk}\left(\UnitC{k}{k}\right)\phi_{km} = \psi_{kk}*\phi_{km}= \phi_{km}\left(\UnitC{k}{m}\right)\psi_{km} = \psi_{km}.\]

\end{proof} 

By Lemma \ref{LemLeftisRight}, the unicity in Theorem \ref{TheoInvInt} already follows. It implies as well that, by symmetry, it is sufficient to find a left invariant integral for $\mathscr{G}$.

The following lemma will be crucial.

\begin{Lem}\label{LemRefSep} Let $\omega \in \Gr{B}{k}{m}{l}{n}$, and assume $\chi*\omega =  0$, resp. $\omega*\chi= 0$ for all $\chi \in \Gr{B}{m}{k}{n}{l}$. Then $\omega =0$.
\end{Lem} 

\begin{proof} Assume that $\chi*\omega =0$ for all $\chi \in \Gr{B}{m}{k}{n}{l}$.
Then since $\omega = \omega(\UnitC{k}{m}\,\cdot\,\UnitC{l}{n})$, we have for all $\chi\in A^*$ and $a\in A$ that, writing $P= \sum_p \UnitC{p}{p}$, \begin{eqnarray*} \omega((\chi\otimes \id)(\Delta(PaP))) &=& (\chi \otimes \omega)((\sum_{k',m'}\UnitC{m'}{k'}\otimes \UnitC{k'}{m'}))\Delta(a)(\sum_{l',n'}\UnitC{n'}{l'}\otimes \UnitC{l'}{n'})) \\ &=&  (\chi \otimes \omega)\left(\left(\UnitC{m}{k}\otimes \UnitC{k}{m}\right)\Delta(a)\left(\UnitC{n}{l}\otimes \UnitC{l}{n}\right)\right) \\
&=&  \left(\chi\left(\UnitC{m}{k}\,\cdot\,\UnitC{n}{l}\right)*\omega\right)(a) \\
  &=&0,\end{eqnarray*} 
since $\chi\left(\UnitC{m}{k}\,\cdot\,\UnitC{n}{l}\right) \in \Gr{B}{m}{k}{n}{l}$. By condition (D2) in Definition \ref{DefCpcqg}, we conclude $\omega=0$. 

By a similar argument, we have that $\omega*\chi= 0$ for all $\chi \in \Gr{B}{m}{k}{n}{l}$ implies $\chi =0$. 
\end{proof}

\begin{Rem}\label{RemEquiv} Assume that $(A,\Delta,\{\UnitC{k}{l}\})$ satisfies all axioms for a C$^*$-algebraic partial compact quantum group, except possibly (D2). Assume however that the conditions in Lemma \ref{LemRefSep} hold. Then we claim that $(A,\Delta,\{\UnitC{k}{l}\})$ is a C$^*$-algebraic partial compact quantum group. Indeed, by the Hahn-Banach theorem, Lemma \ref{LemRefSep} entails that, for all $k,l,m,n\in I$, \[\lbrack (\chi\otimes \id)\Delta(A_P) \mid \chi\in \Gr{B}{m}{k}{n}{l} \rbrack = \UnitC{k}{m}A\UnitC{l}{n},\]  from which one half of the condition (D2) immediately follows. The other half follows by symmetry. 

This interpretation makes the density condition (D2) very natural, since eventually one would like $B^*$ to be a $^*$-algebra in which $\omega^**\omega = 0$ implies $\omega =0$. 
\end{Rem}

\begin{Lem}\label{LemRedCorn} Assume that there exists a family of states $\{\phi_{kk}\}$ in $\Gr{B}{k}{k}{k}{k}$ such that, for any $\omega \in \Gr{B}{k}{k}{k}{k}$, one has \[\omega*\phi_{kk} = \omega\left(\UnitC{k}{k}\right)\phi_{kk}.\] Then $(A,\Delta)$ admits a left invariant state.
\end{Lem}
\begin{proof} Let $\theta_{rm}\in \Gr{B}{r}{m}{r}{m}$ be a collection of functionals on $A$ with $\theta_{rm}$ a state whenever $\UnitC{r}{m}\neq \{0\}$, and $\theta_{rm}=0$ otherwise. Write \[\phi_{rm} = \theta_{rm}*\phi_{mm}.\] By assumption, this notation is consistent in the case $r=m$. 

Assume now that $\omega \in \Gr{B}{k}{r}{l}{r}$ and $\chi \in \Gr{B}{m}{k}{m}{l}$. Assume first that $\UnitC{r}{m}\neq 0$ and $\UnitC{k}{m}\neq 0$. Then \begin{eqnarray*} \chi*(\omega*\phi_{rm}) &=& (\chi*\omega*\theta_{rm})*\phi_{mm} \\ &=&  (\chi*\omega*\theta_{rm})(\UnitC{m}{m}) \phi_{mm} \\ &=&  \chi(\UnitC{m}{k})\omega(\UnitC{k}{r})\phi_{mm} \\ &=&  \omega(\UnitC{k}{r}) \; (\chi*\theta_{km})\left(\UnitC{m}{m}\right)\phi_{mm}
\\ &=& \omega(\UnitC{k}{r}) \; (\chi*\theta_{km})*\phi_{mm} \\ &=&  \omega(\UnitC{k}{r}) \; \chi*\phi_{km} .\end{eqnarray*} As $\chi$ was arbitrary, we find by Lemma \ref{LemRefSep} that \begin{equation}\label{EqInvL2} \omega*\phi_{rm} =  \omega(\UnitC{k}{r}) \phi_{km}.\end{equation}

Assume now that $\UnitC{r}{m}=0$. By Lemma \ref{LemEqRel}, $\UnitC{r}{k}=0$ or $\UnitC{k}{m}=0$. Again by Lemma \ref{LemEqRel}, either $\UnitC{k}{r}=0$ or $\UnitC{k}{m}=0$. In either case, both sides of \eqref{EqInvL2} are zero. 

Similarly, if $\UnitC{k}{m}=0$, we conclude that either $\UnitC{k}{r}=0$ or $\UnitC{r}{m}=0$, and again both sides of \eqref{EqInvL2} are zero.

This shows that \eqref{EqInvL2} holds for all indices, and hence $\{\phi_{km}\}$ is a left invariant integral. 
\end{proof} 

Theorem \ref{TheoInvInt} will now be proven once we can produce a family of invariant states $\phi_{kk}$ as in Lemma \ref{LemRedCorn}. For this, one can follow the proof as in \cite{MVD1} for the existence of an invariant state on a compact quantum group.

\begin{Prop}\label{PropInvPart} For each $k\in I$, there exists a state $\phi_{kk}$ in $\Gr{B}{k}{k}{k}{k}$ such that, for any state $\omega \in \Gr{B}{k}{k}{k}{k}$, one has \[\omega*\phi_{kk} =\phi_{kk} = \phi_{kk}*\omega.\]
\end{Prop} 
\begin{proof} Let $k\in I$, and $\omega$ a state in $\Gr{B}{k}{k}{k}{k}$. By taking a limit of Ces\`{a}ro sums of iterated convolutions of $\omega$ as in \cite[Lemma 4.2]{MVD1},
  there exists a state $h_{kk} \in \Gr{B}{k}{k}{k}{k}$ with \begin{equation}\label{EqLocalInv}\omega *h_{kk}= h_{kk} =
  h_{kk}*\omega.\end{equation}

  Assume now that $\rho\in \Gr{B}{k}{k}{k}{k}$ and $0\leq \rho\leq \omega$. Take $a\in A$. Then the
  beginning of the proof of \cite[Lemma 4.3]{MVD1}, applied with $b= (\id\otimes h_{kk})\Delta(a)$,
  shows that, for all $c\in A$, 
  \begin{equation}\label{EqAuxI3d} (h_{kk}\otimes
    (\rho*h_{kk}))((c\otimes 1)\Delta(a)) = \rho(1) (h_{kk}\otimes h_{kk})((c\otimes
    1)\Delta(a)).\end{equation} 	Indeed, this part of the proof only relies on $\Delta$ being a coassociative $^*$-homomorphism, and $h_{kk}$ and $\omega$ being states satisfying \eqref{EqLocalInv}. 
    
But since now $(A\otimes A)\Delta(1) = [(A\otimes 1)\Delta(A)]$, we may
  replace $(c\otimes 1)\Delta(a)$ with $\UnitC{k}{k}\otimes d$ for $d\in \Gr{A}{k}{k}{k}{k}$. Then
  \eqref{EqAuxId} becomes $(\rho*h_{kk})(d) = \rho(1)h_{kk}(d)$. Hence \[\rho*h_{kk} =
  \rho(1)h_{kk}.\] By symmetry, also $(\rho*h_{kk}) = \rho(1)h_{kk}$.
  
We can now conclude the proof by a compactness argument as in \cite[Theorem 4.4]{MVD1}. Indeed, for $\omega$ a positive functional in $\Gr{B}{k}{k}{k}{k}$, let \[K_{\omega} = \{h\in \Gr{B}{k}{k}{k}{k}\mid h\textrm{ state},h*\omega = \omega*h = \omega(1)h\}.\] Then the $K_{\omega}$ are non-empty compact subsets of $A^*$, with non-trivial finite intersections since $K_{\omega_1+\omega_2}\subseteq K_{\omega_1}\cap K_{\omega_2}$ by the previous paragraph. We can hence take $\phi_{kk}$ as in the statement of the proposition to be an element in the joint intersection of all $K_{\omega}$. 
\end{proof}

\begin{proof}[Proof (of Theorem \ref{TheoInvInt})] We simply combine Proposition \ref{PropInvPart} with Lemma \ref{LemRedCorn} and Lemma \ref{LemLeftisRight}.
\end{proof}

\section{C$^*$-algebraic partial compact matrix pseudogroups}\label{SecPCMPG}

Definition \ref{DefCpcqg} is a generalisation of the most general notion of compact quantum group, as it appears in \cite{Wor98}, see also \cite[Definition 3.4]{MVD1}. In practice however, `atomic' examples are more easily provided by the more restrictive notion of \emph{compact matrix pseudogroup} \cite{W1}. Definition \ref{DefCpcmp} mimics this special case in the partial setting. Before we come to that, we make the following definition. We continue to use the notation introduced in Section \ref{SectionRepSemi} and Definition \ref{Defstproj}.

\begin{Def} Let $\mathscr{G}= (A,\Delta,\{\UnitC{k}{l}\})$ be a C$^*$-algebraic partial compact quantum semigroup. A representation $X$ of $\mathscr{G}$ on an $I$-bigraded Hilbert space $\Hsp$ is called \emph{unitary} if  
 \begin{align} \label{eq:corep-pi}
       X^{*}X= \sum_{n}\rho_{n} \otimes \rho^{\Hsp}_{n} \quad \text{and}
       \quad XX^{*} = \sum_{k} \lambda_{k} \otimes \lambda^{\Hsp}_{k}.
     \end{align}
\end{Def}

\begin{Rems}
\begin{enumerate}
\item Note that the $I^{2}$-grading on $\Hsp$ in condition (Co2) of Definition \ref{def:corepresentation}  is in this case uniquely
determined by $X$.
\item In the classical case, Example \ref{ExaClasRep}, unitarity means that each $\pi(g)$ for $g\in \mathscr{G}(k,m)$ is a unitary $\GrDA{\Hsp}{m}{m}\rightarrow \GrDA{\Hsp}{k}{k}$.
\end{enumerate}
\end{Rems}

\begin{Def} \label{DefCpcmp} We call \emph{C$^*$-algebraic $I$-partial compact matrix pseudogroup} a couple consisting of a C$^*$-algebraic partial compact quantum semigroup and a unitary rcfd representation $X$ on an $I$-bigraded Hilbert space $\Hsp$ such that the following conditions are satisfied:
\begin{enumerate}[(a)]
\item[(U3)] $\UnitC{k}{k}\neq 0$ for all $k$.
\item[(D)] With $\mathscr{A}$ the algebra generated by the $\UnitC{k}{l}$ and the matrix coefficients of all \[\Gr{X}{k}{l}{m}{n} = (\UnitC{k}{m}\otimes 1)X(\UnitC{l}{n}\otimes 1)\in A\otimes B(\GrDA{\Hsp}{m}{n},\GrDA{\Hsp}{k}{l}),\] $\mathscr{A}$ is dense in $A$. 
\item[(A)] There exists a linear, anti-multiplicative map $S: \mathscr{A}\rightarrow \mathscr{A}$ such that $S(\UnitC{k}{l}) = \UnitC{l}{k}$ and \begin{equation}\label{eq:antipode-corep} (S\otimes \id)\Gr{X}{k}{l}{m}{n} = \left(\Gr{X}{m}{n}{k}{l}\right)^*.
\end{equation}
\end{enumerate}
\end{Def}

In the following, we will continue to use the notation $\mathscr{A}$ for the associated dense algebra, and $S$ for the associated `antipode' map.

\begin{Rems}
\begin{enumerate}
\item Note that we assume that $X$ is rcfd, so in particular the $\GrDA{\Hsp}{k}{l}$ are finite dimensional. 
\item Note that this definition is a little stronger than the corresponding definition for compact matrix pseudogroup in \cite{Wor98}, where the generating representation is only assumed to be invertible, and where $\mathscr{A}$ is only assumed to be generated as a \emph{$^*$-algebra} by the matrix coefficients of the representation. In practice however, one can always arrange for the generating representation to be unitary and self-dual, the latter being achieved by taking a direct sum with the dual representation.
  \end{enumerate}
\end{Rems} 

\begin{Theorem}\label{Theopseudoquantum} Let $(\mathscr{G},X)$ define a C$^*$-algebraic $I$-partial compact matrix pseudogroup. Then $\mathscr{G}$ is a C$^*$-algebraic $I$-partial compact quantum group.
\end{Theorem} 

We need some preparations.

\begin{Lem} Assume that $X,Y$ are unitary rcfd representations of the C$^*$-algebraic partial compact quantum semigroup $(A,\Delta,\{\UnitC{k}{l}\})$ on respective $I$-bigraded Hilbert spaces $\Hsp$ and $\mathcal{K}$. Then also $X \Circt Y$ is a unitary rcfd representation.
\end{Lem} 
\begin{proof} Immediate, using $(\rho_m\otimes 1)Y =Y(1\otimes \lambda_m^{\mathcal{K}})$. 
\end{proof} 

\begin{Lem} Let $\mathscr{G} =  (A,\Delta,\{\UnitC{k}{l}\})$ be a C$^*$-algebraic partial compact matrix pseudogroup. Assume $X,Y$ are rcfd representations satisfying \eqref{eq:antipode-corep}. Then also $X\Circt Y$ satisfies \eqref{eq:antipode-corep}. 
\end{Lem} 

\begin{proof} Immediate by the anti-multiplicativity of $S$. 
\end{proof}

\begin{Lem} Let $\mathscr{G} = (A,\Delta,\{\UnitC{k}{l}\})$ be a C$^*$-algebraic partial compact quantum semigroup, and $X$ a representation of $\mathscr{G}$ on an $I$-bigraded Hilbert space $\Hsp$. Then \[\pi_X: A^*\rightarrow B(\Hsp),\quad\omega \mapsto (\omega\otimes \id)X\] is a representation of $A^*$ with respect to the convolution product. 
\end{Lem} 
\begin{proof} Immediate by the representation property of $X$. 
\end{proof}

\begin{proof}[Proof (of Theorem \ref{Theopseudoquantum})] We have to prove that the density conditions (D1) and (D2) are satisfied.

Let $n\geq 0$ and $Y = X^{\Circt n}$, where $X^{\Circt 0}$ is considered to be the (obviously unitary) trivial representation, see Example \ref{exa:corep-trivial}. Let $\Hsp_Y = \itimes^n \Hsp$, and choose orthonormal bases $\{e_i\}$ for the components of $\Hsp_Y$. Let $\left(\Gr{Y}{k}{l}{m}{n}\right)_{ij}\in \UnitC{k}{m}A\UnitC{l}{n}$ be the corresponding matrix coefficients. As $Y$ is a unitary representation, we find that, in the strict topology, \begin{eqnarray*} \sum_{p,q,g} \Delta\left(\left(\Gr{Y}{k}{l}{p}{q}\right)_{ig}\right)\left(1\otimes \left(\Gr{Y}{m}{n}{p}{q}\right)_{jg}^*\right) &=&\underset{r,s,h}{\sum_{p,q,g}} \left(\Gr{Y}{k}{l}{r}{s}\right)_{ih}\otimes \left(\Gr{Y}{r}{s}{p}{q}\right)_{hg}\left(\Gr{Y}{m}{n}{p}{q}\right)_{jg}^* 
\\  &=&\left(\Gr{Y}{k}{l}{m}{n}\right)_{ij}\otimes \lambda_m \\&=& \Delta(1)\left(\left(\Gr{Y}{k}{l}{m}{n}\right)_{ij}\otimes 1\right).\end{eqnarray*} Multiplying with $1\otimes \UnitC{m}{p}$, we find \[\sum_{q,g} \Delta\left(\left(\Gr{Y}{k}{l}{p}{q}\right)_{ig}\right)\left(1\otimes \left(\Gr{Y}{m}{n}{p}{q}\right)_{jg}^*\right)= \Delta(1)\left(\left(\Gr{Y}{k}{l}{m}{n}\right)_{ij}\otimes \UnitC{m}{p}\right),\] where now the left hand side is a finite sum by the rcfd condition.

Since $A$ is by definition densily spanned by the matrix coefficients of all $X^{\Circt n}$, it follows that \[\Delta(1)(A\otimes A) = \lbrack \Delta(A)(1\otimes A)\rbrack.\] In a similar way, the other density condition in (D1) is satisfied. 

To verify (D2), it is, by Remark, \ref{RemEquiv} sufficient to check that the conclusion of Lemma \ref{LemRefSep} is satisfied. But take $\omega \in \Gr{B}{k}{m}{l}{n}$ non-zero. Let \[\omega^*(a) = \overline{\omega(S(a)^*)},\qquad a\in\mathscr{A}.\] Then the defining property of $S$ shows that, for any $Y = X^{\Circt n}$, the associated representation $\pi_{Y}$ of $A^*$ on $\Hsp_Y$ satisfies \[\pi_Y(\omega)^*  = (\omega^*\otimes \id)\Gr{Y}{m}{n}{k}{l}.\] As $A$ is densily spanned by the algebra generated by the matrix coefficients of $X$, it follows that we can take $Y$ such that $\pi_Y(\omega) \neq 0$. Choose then $\chi \in \Gr{B}{m}{k}{n}{l}$ such that $\chi = \omega^*$ on the matrix coefficients of $Y$, which is possible since $Y$ is rcfd. It then follows that \[\pi_Y(\chi*\omega) = \pi_Y(\chi)\pi_Y(\omega) = \pi_Y(\omega)^*\pi_Y(\omega) \neq0.\] Hence $\chi*\omega \neq 0$. By the same argument, $\omega*\chi\neq 0$. It follows that the conclusion of Lemma \ref{LemRefSep} holds.
\end{proof}

\section{A general construction method}

C$^*$-algebraic partial compact matrix pseudogroups can be easily created from algebraic data as follows. 

Note first that the definition of a C$^*$-algebraic $I$-partial compact quantum semigroup still makes sense if $A$ is replaced by a general $^*$-algebra $\mathscr{A}$, once one interprets 
\begin{itemize}
\item `strict convergence of $\sum_{k,l} \UnitC{k}{l}$' as `$\mathscr{A}$ is spanned by its parts $\UnitC{k}{l}\mathscr{A}$'.
\item `$\Delta: A\rightarrow M(A\otimes A)$' as `$\Delta: \mathscr{A}\rightarrow M(\mathscr{A}\underset{\alg}{\otimes} \mathscr{A})$'.
\end{itemize}

The coassociativity condition on $\Delta$ can be made sense of, as one now has the equality $\Delta(\mathscr{A})(\mathscr{A}\otimes \mathscr{A}) = \Delta(1)(\mathscr{A}\otimes \mathscr{A})$, so there is a unique `continuous' extension of (for example) $\Delta$ to the multiplier $^*$-algebra $M(\mathscr{A})$, such that $\Delta$ sends the unit to $\Delta(1)$.

We will call the above algebraic structures \emph{$^*$-algebraic $I$-partial compact quantum semigroups}. Note that also the elements $\lambda_r$ and $\rho_r$ of Definition \ref{Defstproj} still make sense inside $M(\mathscr{A})$. We will further write \[\Gr{\mathscr{A}}{k}{l}{m}{n} = \UnitC{k}{m}\mathscr{A}\UnitC{l}{n}.\]

Also the notion of an rcfd representation still makes sense for $^*$-algebraic partial compact quantum semigroups. To avoid awkward technicalities, we rephrase the definition in the following form. First, for $r,s\in I$, let us write \[\Delta_{rs}(a) = (\rho_r\otimes 1)\Delta(a)(\rho_s\otimes 1) = (1\otimes \lambda_r)\Delta(a)(1\otimes \lambda_s)\in \mathscr{A}\aotimes \mathscr{A},\] which is indeed an element in the algebraic tensor product since one may assume for example $a\in \Gr{\mathscr{A}}{k}{l}{m}{n}$. 

\begin{Def} An \emph{rcfd representation} of a $^*$-algebraic partial compact quantum semigroup $(\mathscr{A},\Delta,\{\UnitC{k}{l}\})$ consists of an rcfd $I$-bigraded Hilbert space $\Hsp$ and elements \[\Gr{X}{k}{l}{m}{n} \in \Gr{\mathscr{A}}{k}{l}{m}{n}\otimes B(\GrDA{\Hsp}{m}{n},\GrDA{\Hsp}{k}{l})\] such that \[(\Delta_{rs}\otimes \id)(\Gr{X}{k}{l}{m}{n}) = \left(\Gr{X}{k}{l}{r}{s}\right)_{13}\left(\Gr{X}{r}{s}{m}{n}\right)_{23}.\]
It is called \emph{unitary} if \[\sum_k \left(\Gr{X}{k}{l}{m}{n'}\right)^*\Gr{X}{k}{l}{m}{n} = \delta_{n,n'} \UnitC{l}{n}\otimes \id_{\GrDA{\Hsp}{m}{n}},\]\[ \sum_{n} \Gr{X}{k}{l}{m}{n}\left(\Gr{X}{k'}{l}{m}{n}\right)^* = \delta_{k,k'} \UnitC{k}{m}\otimes \id_{\GrDA{\Hsp}{k}{l}}.\]
\end{Def} 

Note that the sums in the unitarity condition are in fact finite, by the rcfd condition.

As for C$^*$-algebraic partial compact quantum semigroups, one can define tensor products of (unitary) rcfd representations of $^*$-pcqsg.

The following definition is now obvious. 

\begin{Def} \label{Defalgpcmp} We call \emph{$^*$-algebraic $I$-partial compact matrix pseudogroup} a couple consisting of a $^*$-algebraic partial compact quantum semigroup and a unitary rcfd representation $X$ on an $I$-bigraded Hilbert space $\Hsp$ such that the following conditions are satisfied:
\begin{enumerate}[(a)]
\item[(U3)] $\UnitC{k}{k}\neq 0$ for all $k$.
\item[(G)] $\mathscr{A}$ is generated as an algebra by the $\UnitC{k}{l}$ and the matrix coefficients of the $\Gr{X}{k}{l}{m}{n}$.
\item[(A)] there exists a linear, anti-multiplicative map $S: \mathscr{A}\rightarrow \mathscr{A}$ such that $S(\UnitC{k}{l}) = \UnitC{l}{k}$ and \begin{equation}\label{eq:antipode-corepalg} (S\otimes \id)\Gr{X}{k}{l}{m}{n} = \left(\Gr{X}{m}{n}{k}{l}\right)^*.
\end{equation}
\end{enumerate}
\end{Def}  

\begin{Theorem} Let $(\mathscr{A},\Delta,\{\UnitC{k}{l}\})$ be a $^*$-algebraic $I$-partial compact matrix pseudogroup with generating unitary representation $X$ on an $I$-bigraded Hilbert space $\Hsp$. Then for each $a\in \mathscr{A}$, there exists $M_a\geq 0$ such that $\|\pi(a)\|\leq M_a$ for all $^*$-homomorphisms $\pi: \mathscr{A}\rightarrow B(\mathcal{K})$, $\mathcal{K}$ a Hilbert space.
\end{Theorem} 

\begin{proof}  Let $\pi: \mathscr{A}\rightarrow B(\mathcal{K})$ be a $^*$-representation on a Hilbert space $\mathcal{K}$. 

By the generating condition, it suffices to prove that $\|\pi(a)\|$ is bounded independently of $\pi$ for  $a$ of the form \[a=(\id \otimes \omega_{\xi,\eta})(\Gr{X}{k}{l}{m}{n}),\]
where $\xi\in \Gru{\Hsp}{k}{l}$, $\eta\in
\Gru{\Hsp}{m}{n}$.  However, by unitarity we have  \begin{equation}\label{EqUnitary}
    \sum_{p}(\Gr{X}{p}{l}{m}{n})^{*} \Gr{X}{p}{l}{m}{n}  = \UnitC{l}{n}
    \otimes \id_{\Gru{\Hsp}{m}{n}}.
  \end{equation}

As $\pi(\UnitC{l}{n})$ is a self-adjoint projection, it follows that each $(\pi\otimes \id)(\Gr{X}{p}{l}{m}{n})$ is a contraction in $B(\mathcal{K}\otimes \GrDA{\Hsp}{m}{n},\mathcal{K}\otimes \GrDA{\Hsp}{p}{l})$. We hence conclude that
 \[\|\pi(a)\| \leq \| \xi\| \|\eta\| \|(\pi \otimes \id)(\Gr{X}{k}{l}{m}{n})\| \leq
    \|\xi\|\|\eta\|. \qedhere \]  

\end{proof} 

\begin{Cor}  Let $(\mathscr{A},\Delta,\{\UnitC{k}{l}\})$ be a $^*$-algebraic $I$-partial compact matrix pseudogroup. Then \[\|a\|_u = \sup\{\|\pi(a)\|\mid \pi \textrm{ $^*$-representation of }A\}\] defines a seminorm on $\mathscr{A}$, and, with $J= \{a\in \mathscr{A} \mid \|a\|_u=0\}$, the completion of $\mathscr{A}/J$ with respect to $\|\,\cdot\,\|_u$ is a C$^*$-algebra $A$.
\end{Cor} 

We will call $A$ the \emph{universal C$^*$-envelope} of $\mathscr{A}$, although in general the natural map from $\mathscr{A}$ into $A$ will not be injective!

\begin{Theorem}\label{TheoAlgtoAn} Let $(\mathscr{A},\Delta,\{\UnitC{k}{l}\})$ be a $^*$-algebraic $I$-partial compact matrix pseudogroup with generating representation $X$. Let $A$ be the universal C$^*$-envelope of $\mathscr{A}$, with associated $^*$-homomorphism \[\pi_u: \mathscr{A}\rightarrow A.\] Assume that $\pi_u$ is injective. Then the comultiplication on $\mathscr{A}$ descends and extends to a comultiplication on $A$, making $(A,\Delta,\{\UnitC{k}{l}\})$ into a C$^*$-algebraic $I$-partial compact matrix pseudogroup over $I$ with generating representation $X = \sum_{k,l,m,n} (\pi\otimes \id)\Gr{X}{k}{l}{m}{n}$.
\end{Theorem} 

\begin{Rem} By the same reasoning as in Section \ref{SecPCMPG}, one can show that, even if $\pi_u$ is not injective, $(A,\Delta,\{\UnitC{k}{l}\})$ is well-defined as a C$^*$-algebraic partial compact quantum group over $I' = \{k\in I\mid \pi(\UnitC{k}{k})\neq 0\}$. It is however, by condition (A), not immediately clear that this is then a C$^*$-partial compact matrix pseudogroup over $I'$. One can prove that this is the case, but this will be treated in more detail elsewhere (see \cite{DCT1}). 
\end{Rem}

\begin{proof}[Proof (of Theorem \ref{TheoAlgtoAn})] By the universal property of $(A,\pi_u)$, we can extend $\Delta$ to a $^*$-homomorphism from $A$ to $M(A\otimes A)$. This obviously makes $(A,\Delta,\{\UnitC{k}{l}\})$ into a C$^*$-algebraic $I$-partial compact quantum semigroup. Trivially, $X =  \sum_{k,l,m,n} (\pi\otimes \id)\Gr{X}{k}{l}{m}{n}$ is well-defined as a strict limit, since it is equivalent to a direct sum of contractive maps. It is immediately clear that this makes $(A,\Delta,\{\UnitC{k}{l}\})$ into a C$^*$-algebraic $I$-partial compact matrix pseudogroup  with generating unitary representation $X$.
\end{proof} 

\section{Dynamical quantum $SU(2)$ group}
  
Dynamical quantum groups were introduced in \cite{EtV1}, and the specific example of dynamical quantum $SU(2)$ was treated in detail in \cite{KoR1}. This dynamical quantum $SU(2)$-group can be seen as a quantization-deformation of an $\R\times \R$-field of $SU(2)$-groups, with a global Poisson structure making the field into a Poisson groupoid, but \emph{not} a field of Poisson groups.   

These dynamical quantum groups were treated in \cite{EtV1,KoR1} within a purely algebraic framework. We will show here that dynamical quantum $SU(2)$ also has an operator algebraic implementation within the context of C$^*$-algebraic partial compact quantum groups. In fact, it is a specific example of the class of examples developed in \cite[Section 5]{DCT2}, whose connection to C$^*$-algebraic partial compact quantum groups will be explained in detail in \cite{DCT1}. However, the case of dynamical quantum $SU(2)$ can be treated more directly within the formalism developed in this paper.

Unlike the algebraic case treated in \cite{EtV1,KoR1}, our dynamical quantum $SU(2)$-groups will depend, apart from the $q$-parameter, on an extra $x$-parameter.

Fix $0<q<1$ and $x>0$. Let \[\Lambda_{q,x} = \Lambda_x = \Lambda = xq^{\Z},\] and let $\mathscr{B}_{q,x} = \mathscr{B}_x = \mathscr{B}$ be the $^*$-algebra of finite support functions on $\Lambda\times \Lambda$. We write the Dirac functions in $\mathscr{B}$ as $\delta_{(y,z)}=\UnitC{y}{z}$.

The following functions will be repeatedly used, \[\ctau(y) = y+y^{-1},\qquad w_{\pm}(y) = \frac{\ctau(q^{\pm 1}y)}{\ctau(y)}.\]

\begin{Def}\label{DefDynSUq2} We define $\mathscr{A}_{q,x} = \mathscr{A}_x = \mathscr{A}$ to be the $^*$-algebra generated by a copy of $\mathscr{B}$ and elements \[u_{\epsilon,\nu;y,z}\] for $\epsilon,\nu\in \{-1,1\}=\{-,+\}$ and $y,z\in \Lambda$ with defining relations \begin{eqnarray}  u_{\epsilon,\nu;y,z}&\in&\Gr{\mathscr{A}}{y}{q^{\epsilon}y}{z}{q^{\nu}z}, \\ \label{EqOrt1} \sum_{\mu\in \{\pm\}} u_{\mu,\epsilon;q^{-\mu}w,y}^* u_{\mu,\nu;q^{-\mu}w,z}&=& \delta_{\epsilon,\nu} \delta_{y,z} \UnitC{w}{q^{\epsilon}y},\\\label{EqOrt2} \sum_{\mu\in \{\pm\}} u_{\epsilon,\mu;y,w} u_{\nu,\mu;z,w}^* &=& \delta_{\epsilon,\nu}\delta_{y,z} \UnitC{y}{w} \\ u_{\epsilon,\nu;y,z}^* &=& \frac{\nu w_{\nu}(z)^{1/2}}{\epsilon w_{\epsilon}(y)^{1/2}} u_{-\epsilon,-\nu;q^{\epsilon}y,q^{\nu}z}.\end{eqnarray}
\end{Def}

We want to show that $\mathscr{A}$ can be made into a $^*$-algebraic $\Lambda$-partial compact matrix pseudogroup.

\begin{Lem} There exists a unique $^*$-homomorphism \[\Delta: \mathscr{A}\rightarrow M(\mathscr{A}\otimes \mathscr{A})\] such that $\Delta(\UnitC{y}{z}) = \sum_{v\in \Lambda} \UnitC{y}{v}\otimes \UnitC{v}{z}$ and \[\Delta(u_{\epsilon,\nu;y,z}) = \sum_{\mu,v} u_{\epsilon,\mu;y,v}\otimes u_{\mu,\nu;v,z}.\] Moreover, $(\mathscr{A},\Delta,\{\UnitC{y}{z}\})$ becomes in this way a $^*$-algebraic $\Lambda$-partial compact quantum semigroup.
\end{Lem} 

\begin{proof} It is easily checked that the images under $\Delta$ of the generators satisfy the same relations. It is then immediate that the resulting structure forms a $^*$-algebraic $\Lambda$-partial compact quantum semigroup. 
\end{proof}

\begin{Lem}\label{LemAlmostpseudo} Consider $\Hsp = l^2(\{-,+\})\otimes l^2(\Lambda)$ as a $\Lambda$-bigraded Hilbert space by the bigrading \[e_{\epsilon}\otimes e_y \in \GrDA{\Hsp}{y}{q^{\epsilon}y}.\] 
For $\epsilon,\nu\in \{-,+\}$ and $y,z\in \Lambda$, put \[\Gr{X}{y}{q^{\epsilon}y}{z}{q^{\nu}z} = u_{\epsilon,\nu;y,z} \otimes e_{\epsilon,\nu}\otimes e_{y,z},\] and put all other expressions $\Gr{X}{y}{y'}{z}{z'}$ equal to zero. Then $X$ is a unitary rcfd representation, w.r.t. which $(\mathscr{A},\Delta,\{\UnitC{k}{l}\})$ satisfies conditions (G) and (A) for a $^*$-algebraic $\Lambda$-partial compact matrix pseudogroup.
\end{Lem} 
\begin{proof} The unitarity of $X$ is just a rephrasing of the orthogonality relations \eqref{EqOrt1} and \eqref{EqOrt2} in Definition \ref{DefDynSUq2}. The representation property of $X$ is immediate from the definition of $\Delta$, and the rcfd condition is immediate from the structure of the bigrading on $\Hsp$. 

It remains to verify (G) and (A) in Definition \ref{Defalgpcmp}. 

Condition (G) is immediate by construction.

For condition (A), one verifies by direct computation that the assignment \[S(u_{\epsilon,\nu;y,z}) = u_{\nu,\epsilon;z,y}\] extends to a linear anti-homomorphism satisfying the requirements in condition (A). 
 \end{proof} 

To finish proving that $(\mathscr{A},\Delta,\{\UnitC{k}{l}\})$ is a $^*$-algebraic $\Lambda$-partial compact matrix pseudogroup, we need to show that none of the $\UnitC{y}{y}$ are zero. We will combine this with proving that  $\mathscr{A}$ has a large enough C$^*$-envelope $A$, that is, that $\mathscr{A}$ embeds into $A$. For this, the precise form of the functions $w_{\epsilon}$ will be needed. 

To prepare this proof, we first find a presentation of $\mathscr{A}$ in terms of certain multiplier elements, which will also make clearer the connection with the approach to dynamical quantum $SU(2)$ in \cite{KoR1}. 

For a function $f$ on $\Lambda\times \Lambda$, write \[f(\lambda,\rho) = \sum_{y,z} f(y,z)\UnitC{y}{z} \in M(\mathscr{A}).\] Similarly, for a function $f$ on $\Lambda$ we write \[f(\lambda) = \sum_{y,z} f(y)\UnitC{y}{z},\qquad f(\rho) = \sum_{y,z}f(z)\UnitC{y}{z}.\] We then write for example $f(q\lambda,\rho)$ for the element corresponding to $(y,z)\mapsto f(qy,z)$.

We can further form in $M(\mathscr{A})$ the elements $u_{\epsilon,\nu} = \sum_{y,z} u_{\epsilon,\nu;y,z}$. Then $u=(u_{\epsilon,\nu})$ is a unitary 2$\times$2 matrix. Moreover, \begin{equation}\label{EqAdju}u_{\epsilon,\nu}^* = u_{-\epsilon,-\nu}\frac{ \nu w_{\nu}^{1/2}(\rho)}{\epsilon w_{\epsilon}^{1/2}(\lambda)}.\end{equation} We have the following commutation relations between functions on $\Lambda\times \Lambda$ and the entries of $u$: \begin{equation}\label{EqGradu} f(\lambda,\rho)u_{\epsilon,\nu} = u_{\epsilon,\nu}f(q^{-\epsilon}\lambda,q^{-\nu}\rho).\end{equation}

\begin{Rem}The comultiplication on $\mathscr{A}$ satisfies 
\[\Delta(u_{\epsilon,\nu}) = \Delta(1)\left(\sum_{\mu} u_{\epsilon,\mu}\otimes u_{\mu,\nu}\right).\]
\end{Rem}

In the following, we will write $u_{--}=\alpha, u_{-+}= \beta, u_{+-}=\gamma,u_{++}=\delta$. These satisfy the following relations.

\begin{equation}\label{EqId1}\left\{\begin{array}{lllllll} \alpha\alpha^* + \beta\beta^*, &=& 1 &&  \gamma\gamma^* + \delta\delta^* &=& 1,\\ \alpha^*\alpha+ \gamma^*\gamma &=&1,&&\beta^*\beta+ \delta^*\delta &=& 1,\\ \\ \alpha \gamma^* = -\beta \delta^*, &&&& \alpha^*\beta = -\gamma^*\delta, \end{array}\right.\end{equation}

\begin{equation}\label{EqId2} \delta^* = \alpha \frac{w_+^{1/2}(\rho)}{w_+^{1/2}(\lambda)}, \quad \gamma^*=  - \beta\frac{w_{-}^{1/2}(\rho)}{w_+^{1/2}(\lambda)},\quad  \beta^* = - \frac{w_+^{1/2}(\lambda)}{w_{-}^{1/2}(\rho)}\gamma, \quad  \alpha^* =  \frac{w_+^{1/2}(\lambda)}{w_+^{1/2}(\rho)}\delta.\end{equation}

Up to a rescaling and a reinterpretation of the paramater domain for $\lambda$ and $\rho$, these are precisely the commutation relations for dynamical quantum $SU(2)$ as in \cite{KoR1}. 

The identities in the next lemma follow immediately from \eqref{EqId1} and \eqref{EqId2}.

\begin{Lem}\label{LemExtCom} The following identities hold in $M(\mathscr{A})$. \begin{eqnarray}\label{EqId3} \tau(\lambda)\tau(q\rho)\alpha^*\alpha -\tau(\lambda/q)\tau(\rho)\alpha\alpha^* &=& (q-q^{-1})(\lambda\rho-1/\lambda\rho) \\
 \label{EqId4} \tau(\lambda)\tau(\rho/q) \beta^*\beta - \tau(\lambda/q)\tau(\rho) \beta\beta^* &=& (q-q^{-1})(\lambda/\rho-\rho/\lambda) \\  
  \tau(\lambda)\tau(q\rho)\gamma^*\gamma -\tau(q\lambda)\tau(\rho)  \gamma\gamma^* &=& (q-q^{-1})(\lambda/\rho-\rho/\lambda) \\
 \tau(\lambda)\tau(\rho/q)\delta^*\delta - \tau(q\lambda)\tau(\rho)\delta\delta^*   &=&  (q^{-1}-q)(\lambda\rho-1/\lambda\rho) 
 \end{eqnarray}

\end{Lem}

\begin{Theorem} $(\mathscr{A},\Delta,\{\UnitC{k}{l}\})$ is a $^*$-algebraic $\Lambda$-partial compact matrix pseudogroup with generating unitary rcfd representation $X$, and $\mathscr{A}$ embeds faithfully into its universal C$^*$-algebra $A$.
\end{Theorem}
\begin{proof} Let \[E = \{\alpha^k\beta^l\gamma^m\UnitC{y}{z}, \delta^k\beta^l\gamma^m\UnitC{y}{z}, \mid k,l,m\in \N,y,z \in \Lambda\}.\] From the commutation relations for $\mathscr{A}$ in \eqref{EqGradu}, \eqref{EqId1}, \eqref{EqId2} and Lemma \ref{LemExtCom}, it follows immediately that $E$ is a spanning set for $\mathscr{A}$. 

Choose now $-2<c<2$, and define the following operators on $l^2(\Lambda)\otimes l^2(\Lambda)$,

\begin{eqnarray*} \pi_c\left(\UnitC{y'}{z'}\right)e_{y}\otimes e_{z} &=&  \delta_{y,y'}\delta_{z,z'} e_{y}\otimes e_{z},\\ \pi_c(u_{\epsilon,\nu})e_{y}\otimes e_{z} &=& \theta_{\epsilon,\nu}  \left(\frac{\tau(q^{-1}y^{\epsilon}z^{\nu})+\epsilon \nu c}{\tau(y^{\epsilon})\tau(q^{-1}z^{\nu})}\right)^{1/2}e_{q^{-\epsilon}y}\otimes e_{q^{-\nu}z},
\end{eqnarray*}
where $\theta_{-,+}=-1$ and all other values $=1$. 

Note that the operators $\pi_c(u_{\epsilon,\nu})$ are well-defined and bounded, since $\tau(y)\geq 2$ for all positive $y$, and $\tau(yz)+2\leq \tau(y)\tau(z)$ for all positive $y,z$. Also note that, by definition, $\tau(y) = \tau(y^{-1})$. 

Now obviously the $\pi_c(u_{\epsilon,\nu})$ and  $\pi_c\left(\UnitC{y'}{z'}\right)$ commute according to \eqref{EqGradu}. On the other hand, \[\pi_c(u_{\epsilon,\nu})^*e_{y,z} =  \theta_{\epsilon,\nu}\left(\frac{\tau(qy^{\epsilon}z^{\nu})+c}{\tau(qy^{\epsilon})\tau(z^{\nu})}\right)^{1/2}e_{q^{\epsilon}y,q^{\nu}z}.\] It follows that $\pi_c$ respects the relation \eqref{EqAdju}. Finally, it is also easily verified from this that $(\pi_c(u_{\epsilon,\nu}))_{\epsilon,\nu}$ is a unitary matrix, using the identity $\tau(yz)+\tau(y^{-1}z) = \tau(y)\tau(z)$ for $y,z>0$. 

From the above, it follows immediately that $\pi_c$ extends to a $^*$-representation of $\mathscr{A}$. Moreover, by looking at the shift components, it is clear that the only linear dependencies between elements in $E$ can occur within the subfamilies \[E_{k,l,y,z} = \{\alpha^k\beta^l(\beta^*\beta)^m\UnitC{y}{z} \mid m\in \N\},\quad k,l\in \Z,y,z \in \Lambda,\] where a negative power is interpreted as taking the adjoint. But let $p$ be a polynomial, and assume that, for all $-2<c<2$, \[\pi_c\left(\alpha^k\beta^lp(\beta^*\beta)\UnitC{y}{z}\right) = 0.\] Then, for all $c$, \[p\left(\frac{\tau(q^{-1}y^{\epsilon}z^{\nu})+\epsilon \nu c}{\tau(y^{\epsilon})\tau(q^{-1}z^{\nu})}\right) =0.\]
It follows that $p$ is zero on some closed interval, and hence $p=0$. 

From the above, we conclude immediately that $\UnitC{y}{y}\neq 0$ for all $y\in \Lambda$, which, combined with Lemma \ref{LemAlmostpseudo}, shows that $(\mathscr{A},\Delta,\{\UnitC{k}{l}\})$ is a $^*$-algebraic $\Lambda$-partial compact matrix pseudogroup with generating unitary rcfd representation $X$. The above also shows immediately that $\mathscr{A}$ imbeds into its universal C$^*$-algebraic envelope. 
\end{proof} 

\begin{Cor} Let $A$ be the universal C$^*$-algebraic envelope of $\mathscr{A}$. Then $A$ obtains the structure of a C$^*$-algebraic $\Lambda$-partial compact matrix pseudogroup with generating unitary rcfd representation $X$.
\end{Cor} 

We will denote $SU_{q,x}^{\dyn}(2) = (A,\Delta,\{\UnitC{y}{z}\})$. The following proposition clarifies the relation between the $SU_{q,x}^{\dyn}(2)$ for different values of $x$.

\begin{Prop} Assume $x_1,x_2>0$, and assume there exists $m\in \Z$ and $\epsilon \in \{-,+\}$ with $x_2= x_1^{\epsilon}q^m$. Then $SU_{q,x_1}^{\dyn}(2) \cong SU_{q,x_2}^{\dyn}(2)$.
\end{Prop} 
\begin{proof} If $x_2 = x_1q^{m}$, then $\Lambda_{x_2} = \Lambda_{x_1}$, and it follows immediately from the definition of $SU_{q,x}^{\dyn}(2)$ that in fact $SU_{q,x_1}^{\dyn}(2) = SU_{q,x_2}^{\dyn}(2)$.

It thus suffices to prove that $SU_{q,x}^{\dyn}(2) \cong SU_{q,x^{-1}}^{\dyn}(2)$ for $x>0$. But this is established by means of the isomorphism \[\UnitC{y}{z}\mapsto \UnitC{y^{-1}}{z^{-1}},\qquad u_{\epsilon,\nu;y,z}\mapsto u_{-\epsilon,-\nu;y^{-1},z^{-1}},\] which is most easily seen to extend to a ($\Delta$-preserving) $^*$-isomorphism using the description of the dynamical quantum $SU(2)$-group in terms of the matrix $(u_{\epsilon,\nu})_{\epsilon,\nu}$ and the functions $f(\lambda,\rho)$. For example, the relation \eqref{EqAdju} is preserved by the above isomorphism since $w_{\epsilon}(\lambda) = w_{\epsilon}(\lambda^{-1})$. 
\end{proof}

\section{Representation theory of the function algebra on dynamical quantum $SU(2)$}

The representation theory of $SU_{q,x}^{\dyn}(2)$ was essentially determined in \cite{DCT2}, where it was shown to coincide with the representation theory of $SU_q(2)$. On the algebraic level, this follows since $SU_{q,x}^{\dyn}(2)$ is a `dynamical' cocycle twist of $SU_q(2)$ \cite{Sto1}. 

Here, we will rather be concerned with the representation theory of the \emph{function algebra} $A$ on $SU_{q,x}^{\dyn}(2)$. That is, we wish to classify the irreducible $^*$-representations of the $^*$-algebra $\mathscr{A}_x$ associated to $SU_{q,x}^{\dyn}(2)$. 

Our method will be based on a \emph{decoupling} of $\mathscr{A}_x$. For this, we first recall the definition of the \emph{quantized enveloping algebra of $\su(1,1)$}. 

\begin{Def} Let $q>0$. The \emph{quantized enveloping algebra $U_q(\su(1,1))$} is the universal unital $^*$-algebra generated by elements $E,F,K,K^{-1}$ satisfying the following commutation rules: 
\begin{itemize}
\item $K^{-1}$ is the inverse of $K$, 
\item $K^* = K$ and $E^* = F$, 
\item $KE = qEK$,
\item $\lbrack F,E\rbrack = \frac{K^2-K^{-2}}{q-q^{-1}}$.
\end{itemize}
\end{Def}

One can turn $U_q(\su(1,1))$ into a Hopf $^*$-algebra, but this extra structure will not be needed. 

It will be convenient to consider a slight variation of $U_q(\su(1,1))$ by formally adding support projections of $K$, along the lines of \cite[Chapter 23]{Lus1}. For $y>0$, we will write \[\Gamma_y = yq^{\frac{1}{2}\Z}.\]

\begin{Def} Let $y>0$. We define $U_q^{\Gamma_y}(\su(1,1))$ to be the (non-unital) $^*$-algebra generated by a copy of the $^*$-algebra of finite support functions on $\Gamma_y$, whose Dirac functions we will write $\Unit_r$, together with elements $E_{r},F_r$ for each $r\in \Gamma_y$, such that $E_{r}^* = F_{r}$ and 
\begin{itemize}
\item $\Unit_{qr}E_{r} = E_{r} = E_{r}\Unit_{r}$, 
\item $\Unit_r F_r = F_r  = F_r\Unit_{qr}$
\item $F_{r}E_{r} - E_{r/q}F_{r/q}  = \frac{r^2-r^{-2}}{q-q^{-1}} \Unit_{r}$.
\end{itemize}
\end{Def} 

\begin{Lem} We have a unique $^*$-homomorphism $U_q(\su(1,1))\rightarrow M(U_q^{\Gamma_y}(\su(1,1)))$ such that \[K \mapsto \sum_{r\in \Gamma_y} r\Unit_r,\quad E\mapsto \sum_{r\in \Gamma_y} E_{r},\quad F\mapsto \sum_{r\in \Gamma_y} F_r. \] Moreover, this $^*$-homomorphism is injective.
\end{Lem} 
\begin{proof} It is clear that the images are well-defined multipliers. It is then immediate that there is a unique $^*$-homomorphism with the above prescribed images.

Let us show that it is injective. Consider the vector space $V$ with basis $\{e_n\mid n\in \N\}$. Then we can represent $U_q^{\Gamma_y}(\su(1,1))$ on $V$ (neglecting the $^*$-structure) by \[\Unit_r e_n = \delta_{r,yq^n}e_n,\quad E_{r} e_n = \delta_{r,yq^n} e_{n+1},\]\[(q-q^{-1})^2 F_{r}e_n = \delta_{r,yq^{n-1}} (q^n-q^{-n})(q^{n-1}y^2-q^{1-n}y^{-2})e_{n-1}.\] This representation can then be extended to the multiplier algebra. As the elements $E^mK^nF^l$ form a basis of $U_q(\su(1,1))$, it follows that the corresponding representation is injective when considered as a $U_q(\su(1,1))$-represention via the $^*$-homomorphism in the lemma, which implies the injectivity of the map $U_q(\su(1,1))\rightarrow M(U_q^{\Gamma_y}(\su(1,1)))$.
\end{proof}  

In what follows, we will need two copies of $U_q^{\Gamma_y}(\su(1,1))$'s, for different values of $y$. We will correspondingly use the indices $(1)$  and $(2)$ as upper indices for the generators of the two copies.
\begin{Lem}\label{LemHomUqtodyn}
There is a unique non-degenerate $^*$-homomorphism \[\Phi: U_q^{\Gamma_x}(\su(1,1))\otimes U_q^{\Gamma_1}(\su(1,1))\rightarrow \mathscr{A}_x\] such that 
\[(q^{-1}-q)\Phi(E^{(1)})=  \alpha \tau^{1/2}(\lambda)\tau^{1/2}(q\rho),\qquad (q^{-1}-q)\Phi( E^{(2)}) =  \beta \tau^{1/2}(\lambda)\tau^{1/2}(\rho/q),\]
\[(q^{-1}-q)\Phi(F^{(1)})=  \delta\tau^{1/2}(\lambda)\tau^{1/2}(\rho/q),\qquad (q^{-1}-q)\Phi( F^{(2)}) =  - \gamma\tau^{1/2}(\lambda)\tau^{1/2}(q\rho),\]
\[\Phi(\Unit_r^{(1)}) = \sum_{z\in \Lambda_x} \UnitC{r^2/z}{z},\qquad \Phi(\Unit_s^{(2)}) = \sum_{z\in \Lambda_x} \UnitC{s^2z}{z}.\] Moreover, $\Phi$ is surjective.
\end{Lem} 

\begin{proof} 
Note first that the values for $\Phi(E^{(i)})$ and $\Phi(F^{(i)})$ given above are well-defined inside $M(\mathscr{A}_x)$, with $\Phi(F^{(i)})^* = \Phi(E^{(i)})$. We can then define $\Phi(E^{(i)}_r) = \Phi(E^{(i)})\Phi(\Unit_r^{(i)})$ and $\Phi(F^{(i)}_r) = \Phi(\Unit_r^{(i)})\Phi(F^{(i)})$. We have to verify that the $\Phi(\Unit_r^{(i)})$, $\Phi(E^{(i)}_r)$ and $\Phi(F^{(i)}_r)$ satisfy the relations of $U_q^{\Gamma_i}(\su(1,1))$, and that these elements pairwise commute for different values of $i$.

It is  easily seen that $\Phi(\Unit_{qr}^{(i)})\Phi(E^{(i)}_r) = \Phi(E^{(i)}_r) = \Phi(E^{(i)}_r)\Phi(\Unit_r^{(i)})$ and $\Phi(\Unit_r^{(i)})\Phi(E^{(i+1)}_s) = \Phi(E^{(i+1)}_s)\Phi(\Unit_r^{(i)})$ (with the upper indices taken modulo 2).

Let us now verify that $\lbrack \Phi(E^{(i)})^*,\Phi(E^{(i)})\rbrack\Unit_r = \frac{r^2-r^{-2}}{q-q^{-1}}\Phi(\Unit_r)$. This follows immediately from \eqref{EqId3} and \eqref{EqId4}.

It remains to check that $\lbrack \Phi(E^{(1)}),\Phi(E^{(2)})\rbrack = 0$ and $\lbrack \Phi(F^{(1)}),\Phi(E^{(2)})\rbrack = 0$, but these identities are equivalent with the last two identities in \eqref{EqId1}.

Finally, $\Phi(\Unit_r^{(1)}\Unit_s^{(2)}) = \UnitC{rs}{r/s}$ whenever $r/s\in \Lambda_x$, and is zero otherwise. It follows that the range of $\Phi$ takes values in $\mathscr{A}_x$, and moreover that the image of $\Phi$ contains all finite support functions on $\Lambda_x\times \Lambda_x$. It is then clear that in fact $\Phi$ is surjective. 
\end{proof}

\begin{Def}We write \[C = (q^{-1}-q)^2FE-qK^2-q^{-1}K^{-2} \in U_q(\su(1,1))\] for the \emph{Casimir element} of $U_q(\su(1,1))$.
\end{Def}
It is well-known and easily verified that the Casimir element $C$ is a self-adjoint element in the center of $U_q(\su(1,1))$ (and which, in fact, generates the center). 

\begin{Lem}\label{LemKer} The kernel of the map $\Phi$ of Lemma \ref{LemHomUqtodyn} is generated by the collection of all $(C^{(1)}+C^{(2)})\Unit_r^{(1)}\Unit_s^{(2)}$ with $r,s\in \Gamma_x$, and all $\Unit_r^{(1)}\Unit_s^{(2)}$ with $rs\notin \Lambda_x$
\end{Lem} 
\begin{proof} Let $\widetilde{B}_x$ be the quotient of $B_x = U_q^{\Gamma_x}(\su(1,1))\otimes U_q^{\Gamma_1}(\su(1,1))$ by the ideal generated by $(C^{(1)}+C^{(2)})B_x$ and all $\Unit_r^{(1)}\Unit_s^{(2)}$ with $rs\notin \Lambda_x$. It is straightforward to compute, using the commutation relations \eqref{EqId1}, that $\Phi$ descends to a homomorphism $\widetilde{\Phi}$ on $\widetilde{B}_x$. 

Denote the images of the generators of $B_x$ in $\widetilde{B}_x$ by the same symbols adorned with a tilde. For $w,z\in \Lambda_x$, write $\widetilde{\Psi}(\UnitC{w}{z}) = \widetilde{\Unit}_{\sqrt{wz}}^{(1)}\widetilde{\Unit}_{\sqrt{w/z}}^{(2)}$, and \begin{eqnarray*} 
\widetilde{\Psi}(\alpha) &=& (q^{-1}-q)\widetilde{E}^{(1)}\tau^{-1/2}(\widetilde{K}^{(1)}\widetilde{K}^{(2)})\tau^{-1/2}(q\widetilde{K}^{(1)}/\widetilde{K}^{(2)}),\\ 
\widetilde{\Psi}(\beta) &=& (q^{-1}-q)\widetilde{E}^{(2)}\tau^{-1/2}(\widetilde{K}^{(1)}\widetilde{K}^{(2)})\tau^{-1/2}(\widetilde{K}^{(1)}/q\widetilde{K}^{(2)}),\\ 
\widetilde{\Psi}(\gamma) &=& (q-q^{-1})\widetilde{F}^{(2)}\tau^{-1/2}(\widetilde{K}^{(1)}\widetilde{K}^{(2)})\tau^{-1/2}(q\widetilde{K}^{(1)}/\widetilde{K}^{(2)}),\\
\widetilde{\Psi}(\delta) &=&  (q^{-1}-q)\widetilde{F}^{(1)}\tau^{-1/2}(\widetilde{K}^{(1)}\widetilde{K}^{(2)})\tau^{-1/2}(\widetilde{K}^{(1)}/q\widetilde{K}^{(2)}).
\end{eqnarray*}
By a direct computation, these elements satisfy the defining relations \eqref{EqId1} and \eqref{EqId2}. This provides a unique non-degenerate $^*$-homomorphism $\widetilde{\Psi}: \mathscr{A}_x\rightarrow \widetilde{B}_x$ with the above images on generators, forming an inverse to $\widetilde{\Phi}$.
\end{proof}

The irreducible representations of $\mathscr{A}_x$ can now be computed by first classifying the irreducible $^*$-representations of the $U_q^{\Gamma_y}(\su(1,1))$. 

By \emph{$^*$-representation} we will mean a non-degenerate bounded $^*$-representation $\pi$ of a $^*$-algebra on a Hilbert space $\Hsp = \Hsp_{\pi}$. For $U_q^{\Gamma_y}(\su(1,1))$, this means in particular that $\Hsp$ is a (closed) direct sum of the Hilbert spaces $\Hsp_r = \pi(\Unit_r)\Hsp$. We will denote by $H = H_{\pi}$ the \emph{algebraic} direct 
sum of all $\Hsp_r$. This then carries a representation of $U_q(\su(1,1))$. 

The representation theory of $U_q^{\Gamma_y}(\su(1,1))$ is very similar to the one of $U_q(\su(1,1))$ \cite{MMNNSU1}.

\begin{Lem}\label{CorCas} If $\pi$ is an irreducible $^*$-representation of $U_q^{\Gamma_y}(\su(1,1))$, there exists $c\in \R$ such that $\pi(C)\xi = c\xi$ for all $\xi \in H_{\pi}$. 
\end{Lem} 
\begin{proof} As $\pi(C)$ is bounded when restricted to any $\Hsp_r$, this follows immediately from the centrality of $\C$ and a spectral argument. 
\end{proof} 

\begin{Cor}\label{CorOneDim} If $\pi$ is an irreducible $^*$-representation of $U_q^{\Gamma_y}(\su(1,1))$ on a Hilbert space $\Hsp_{\pi}$, then $\Hsp_r$ is at most one-dimensional for each $r\in \Gamma_y$.
\end{Cor} 
\begin{proof} 
Monomials of the form $E^kK^mF^l\Unit_r$ span $U_q^{\Gamma_y}(\su(1,1))$, and hence also the elements of the form $E^kK^mC^l\Unit_r$ combined with those of the form $F^kK^mC^l\Unit_r$. Using Lemma \ref{CorCas}, the corollary then follows immediately from the fact that any non-zero vector in $\Hsp_r$ is cyclic by the irreducibility condition and the fact that $E$ and $F$ unilaterally shift the components in different directions.
\end{proof}

We will classify irreducible $^*$-representations of $U_q^{\Gamma_y}(\su(1,1))$ in terms of what we call \emph{$c$-sets}.

\begin{Def}\label{DefAdapt} Let $c\in\R$. For $\epsilon \in \{\pm\}$, a number $z>0$ will be called \emph{$c_{\epsilon}$-adapted} if \begin{equation}\label{EqAd+}  0 \leq c+ \ctau(q^{-\epsilon}z),\end{equation} and \emph{strictly} $c_{\epsilon}$-adapted if this holds strictly. 

The number $z$ is called \emph{$c$-adapted} if it is both $c_+$- and $c_-$-adapted. 

A subset $Z \subseteq \R_{>0}$ is called a \emph{$c$-set} if the following conditions hold: \begin{itemize} 
\item[$\bullet$] $Z$ is not empty.
\item[$\bullet$] $Z$ consists of $c$-adapted points.
\item[$\bullet$] If $z\in Z$ is strictly $c_{\epsilon}$-adapted, then $q^{-2\epsilon}z \in Z$.
\end{itemize}
A $c$-set is called \emph{irreducible} if it can not be written as the union of two disjoint $c$-sets.
\end{Def}

The $c$-sets in $\R$ can be classified as follows, organized in `series' in analogy with the representations of $U_q(\su(1,1))$. For $z>0$, we write \[Z_z = zq^{2\Z},\quad Z_z^+ = zq^{2\N},\quad Z_{z}^{-} = zq^{-2\N},\quad \textrm{and }Z_1^{0} = \{1\}.\] 

\begin{Prop}\label{PropClass1D} The following lists exhaust all
  irreducible $c$-sets in $\R$.
\begin{enumerate}[a)] 
\item If $c>-2: Z_z$ for $q\leq z< q^{-1}$ (`strange' for $c>2$, `principal' for $|c|\leq 2$),
\item If $c \leq -2$, with $-c=\ctau(w_c)$ with $0 < w_c\leq 1$:
\begin{enumerate}[(i)]
\item\label{It1} $Z_z$ for $\frac{q}{w_c}<z<\frac{w_c}{q}$ (`complementary')
\item\label{Itb} $Z_{qw_c}^+$ and $Z_{1/qw_c}^-$ (`large positive and negative discrete')
\item\label{Ita} Only in the case $w_c>q$: $Z_{q/w_c}^+$ and $Z_{w_c/q}^-$  (`small positive and negative discrete')
\item\label{It2} Only in the case  $w_c= q$: $Z_1^0$ (`trivial')
\end{enumerate}
\end{enumerate}
Furthermore, all the sets within a list for a fixed $c$ are distinct, except that cases \eqref{Itb} and \eqref{Ita} coincide in the case $c=-2$.
\end{Prop} 

\begin{proof} Fix $c\in \R$. 

Assume first that $Z$ is an irreducible $c$-set with $c+\tau(qz)\neq 0$ for all $z\in Z$. It then follows that $Z$ is necessarily invariant under multiplication with $q^{2\Z}$, and from the irreducibility assumption we infer that $Z = zq^{2\Z}=Z_z$ for some $z>0$, which we may choose to be the unique one such that $q\leq z<q^{-1}$. 

If $c>-2$, it is clear that $Z_z$ is a $c$-set for any such $z$. If $c\leq -2$, we may write $c=-w_c-w_c^{-1}$ for a unique $0< w_c\leq 1$. We then have to find a necessary and sufficient condition on $z$ such that $\tau(w_c)\leq \tau(q^{2m+1}z)$ for all $m\in \Z$. Clearly, it is sufficient to have $\tau(w_c)\leq \tau(qz)$ and $\tau(w_c)\leq \tau(z/q)$. Since by assumption $qz\leq 1$ and $z/q\geq 1$, this is equivalent with $qz\leq w_c$ and $\frac{1}{w_c}\leq z/q$, so $\frac{q}{w_c}\leq z\leq \frac{w_c}{q}$. Since by assumption $\tau(w_c)\neq \tau(qz)$ and $\tau(w_c)\neq \tau(z/q)$, we may use strict inequalities.

Assume now that $Z$ is an irreducible $c$-set with $c+\tau(qz)=0$ for some $z\in Z$. Then necessarily we must have $c\leq -2$, and hence $c= -w_c-w_c^{-1}$ for a unique $0< w_c\leq 1$. 

If also $c+\tau(z/q)=0$, then necessarily $z=1$ and $c= -q-q^{-1}$, and we obtain that $Z=\{1\}=Z_1^0$. If $c+\tau(z/q)\neq 0$, then we consider separately the two cases $c+\tau(z/q)>0$ and $c+\tau(z/q)<0$. 

If $c+\tau(z/q)>0$, then we infer that $q^{-2}z\in Z$, and hence $q^{-2\N}z\subseteq Z$. By irreducibility, we infer $q^{-2\N}z = Z= Z_z^-$. We hence have to verify which conditions on $z$ ensure that $q^{-2\N}z$ is an irreducible $c$-set. However, we know already that $\tau(qz)=\tau(w_c)$, hence either $qz = w_c$ or $qz = \frac{1}{w_c}$. In the first case, we obtain on $w_c$ the condition $\tau(w_c)<\tau(w_c/q^2)$, and it is easily seen that this is equivalent with $q<w_c$. In the second case, the inequality $\tau(\frac{1}{w_c})<\tau(\frac{1}{q^2w_c})$ is automatic, for all $w_c\leq 1$.  

The case $c+\tau(z/q)>0$ is similar. 

The above argument classifies all irreducible $c$-sets. The fact that all sets for a fixed $c$ are distinct (except possibly when $c=-2$) is immediately clear. 
\end{proof}

\begin{Def} Fix $y>0$. For $\pi$ an irreducible $^*$-representation of $U_q^{\Gamma_y}(\su(1,1))$, an element $r\in \Gamma_y$ is called \emph{$\pi$-compatible} if $\Hsp_r\neq 0$. 

For $c\in \R$, a subset $T\subseteq \R_{>0}$ is called \emph{$(y,c)$-compatible} if there exists an \emph{irreducible} representation $\pi$ of $U_q^{\Gamma_y}(\su(1,1))$ with $\pi(C) = c$ and $T=\{r\in \Gamma_y\mid \Hsp_{r}\neq \{0\}\}$. In this case, we say that $\pi$ is \emph{$T$-compatible}.
\end{Def}

\begin{Prop}\label{PropClassRep} A set $T\subseteq \R_{>0}$ is a $(y,c)$-compatible set if and only if $Z_T = \{t^2\mid t\in T\}$ is an irreducible $c$-set contained in  $\Lambda_{y^2}$. Moreover, for any $(y,c)$-compatible set $T$ there is exactly one irreducible $^*$-representation $\pi$ of $U_q^{\Gamma_y}(\su(1,1))$, up to unitary equivalence, which is $T$-compatible.
\end{Prop}

\begin{proof} In the proof, we will use the notation $X_+ = E$ and $X_- = F$.

Assume first that $T$ is $(y,c)$-compatible, and let $\pi$ be a $T$-compatible irreducible $^*$-representation of $U_q^{\Gamma_y}(\su(1,1))$. If $r\in T$, then it follows from the definition of the Casimir element that $r^2$ is $c$-adapted. Moreover, if $r\in T$ is strictly $c_{\epsilon}$-adapted, then we have that $\|\pi(X_{\epsilon})\xi\|\neq 0$ for a non-zero $\xi\in \Hsp_r$, hence also $\Hsp_{q^{-\epsilon}r}\neq \{0\}$. It follows that $Z_T$ is a $c$-set. Now if $Z_T=Z_{T_1}\cup Z_{T_2}$ a disjoint union of $c$-sets, it would follow that $\pi$ restricts to the direct sum of all $\Hsp_r$ with $r\in T_1$, contradicting irreducibility. It follows that $Z_T$ is an irreducible $c$-set.

Conversely, let $Z_T$ be an irreducible $c$-set with $T\subseteq \Gamma_y$. Put $\Hsp_{\pi} = l^2(T)$ with the grading determined by $\delta_{r}\in l^2(T)_r$. Define a pair of adjoint operators $\pi(X_{\epsilon})$ on $H_{\pi}$ by the formulae \begin{eqnarray}\label{EqFormRepPlus} (q-q^{-1})\pi(X_{\epsilon})  \delta_{r} &=&  (\ctau(q^{\epsilon}r^2)+c)^{1/2}\delta_{q^{\epsilon}r},\end{eqnarray} where the right hand side is considered as the zero vector when the accompanying scalar factor is zero. Note that the roots on the right hand side are well-defined precisely because $Z_T$ is a $c$-set. 

By direct computation, using the defining commutation relations, we see that $\pi$ defines a $^*$-representation of $U_q^{\Gamma_y}(\su(1,1))$ with $\pi(C) =c$, and clearly this representation is bounded. Moreover, $\pi$ is irreducible since otherwise, by Corollary \ref{CorOneDim}, $T$ would split as a disjoint union of $(y,c)$-compatible sets. Hence $T$ is an $(y,c)$-compatible set.

Now the formula for $\pi(X_+)$ is uniquely determined up to a unimodular gauge factor. As any non-zero $\Hsp_r$ is cyclic for $\pi$, it follows that these gauge factors are determined by their value at one component. We then easily conclude that $\pi$ is in fact the unique $T$-compatible $^*$-representation, up to unitary equivalence.
\end{proof}

If $Z_T$ is an irreducible $c$-set, we will write $\pi_T$ for the accompanying representation of $U_q^{\Gamma_y}(\su(1,1))$. Recall the map $\Phi$ of Lemma \ref{LemHomUqtodyn}.

\begin{Theorem}\label{TheoIrrAx} The irreducible $^*$-representations of $\mathscr{A}_x$ are of the form $\pi_{S,T}$ with \[\pi_{S,T}\circ \Phi = \pi_S\otimes \pi_T,\] where $S,T\subseteq \R_{>0}$ are such that there exists $c\in \R$ with 
\begin{itemize}
\item $Z_S \subseteq \Lambda_{x^2}$ an irreducible $c$-set, 
\item $Z_T \subseteq \Lambda_1$ an irreducible $-c$-set, and 
\item $Z_SZ_T \subseteq x^2q^{2\Z}$. 
\end{itemize}
\end{Theorem}

\begin{proof} By Lemma \ref{LemKer}, any irreducible representation of $\mathscr{A}_x$ is a factorisation over $\Phi$ of some irreducible $^*$-representation $\pi$ of $U_q^{\Gamma_x}(\su(1,1))\otimes U_q^{\Gamma_1}(\su(1,1))$. But as any non-zero $\pi(\Unit_r^{(1)}\Unit_s^{(2)})\Hsp$ is cyclic, it is easily seen that all irreducible $^*$-representations of $U_q^{\Gamma_x}(\su(1,1))\otimes U_q^{\Gamma_1}(\su(1,1))$ split as a tensor product $\pi_1\otimes \pi_2$ of irreducible $^*$-representations. As we want $\pi$ to factor over $\Phi$, we then again infer from Lemma \ref{LemKer} that necessary and sufficient conditions on $\pi_1$ and $\pi_2$ for factorisation over $\Phi$ are that $\pi_1(C)$ and $-\pi_2(C)$ are the same scalar, and $\pi_1(\Unit_r)=0$ or $\pi_2(\Unit_s)=0$ if $rs\notin \Lambda_x$. This is easily seen to be equivalent with the statement of the theorem.
\end{proof}

Let now $\Omega = \Phi(C^{(1)}) = -\Phi(C^{(2)}) \in M(\mathscr{A}_x)$, which is a central element we will call the \emph{Casimir} of $SU_{q,x}^{\dyn}(2)$. It correspons to the Casimir element for dynamical quantum $SU(2)$ introduced in \cite{KoR1}. Let $A_x$ be the universal C$^*$-envelope of $\mathscr{A}_x$. As the $\Omega\UnitC{y}{z}$ define orthogonal bounded elements in $\mathscr{A}_x\subseteq A_x$, we can make sense of $\Omega$ as an element affiliated with $A_x$, i.e. $\Omega\,\eta\,A_x$ \cite{Wor2}.

\begin{Cor}\label{CorSpecUni} Let $k_0\in \Z$ be the unique integer such that $q< q^{k_0}x^2\leq 1$, and write $-c_{0} = \max \{\tau(q^{k_0-1}x^{2}),\tau(q^{k_0}x^{2})\}$. Then the spectrum of $\Omega \eta A_x$ equals the set \[\Spec(\Omega) = \lbrack c_0,q+q^{-1}\rbrack \cup \tau(q^{\Z})\cup \tau (-x^2q^{\Z}).\]
\end{Cor}
\begin{proof} The spectrum of $\Omega$ is the closure of the collection of all values $\Omega$ can take in irreducible $^*$-representations of $\mathscr{A}$. From Theorem \ref{TheoIrrAx}, it follows that $\Spec(\Omega)$ is the collection of all $c$'s such that there exists a $c$-set $Z$ and $-c$-set $W$ with
\begin{itemize}
\item $Z \subseteq \Lambda_{x^2}$,
\item $W\subseteq \Lambda_1$, 
\item $ZW \subseteq x^2q^{2\Z}$.
\end{itemize}

Now from the classification in Proposition \ref{PropClass1D}, it follows that we essentially have to consider 3 cases.

Namely, consider first $-2<c<2$. Then it follows immediately that $Z$ and $W$ always exist, hence $(-2,2)\subseteq \Spec(\Omega)$.

Consider now $c\geq 2$. Then if we find a $-c$-set $W\subseteq \Lambda_1$, the existence of a $Z$ as above is automatically guaranteed. But from the classification in Proposition \ref{PropClass1D}, there are essentially four cases in which such a $W$ can exist. Case \eqref{It1} arises if there exists $m\in \Z$ with $q^{m+1}<w_{-c}$ and $q^{-m+1}<w_{-c}$. Clearly, this happens if and only if $q<w_{-c}\leq 1$, that is, $2\leq c<q+q^{-1}$. By Case \eqref{It2}, we can make the right hand side equality non-strict. On the other hand, it is not hard to see that cases \eqref{Ita} or \eqref{Itb} can appear if and only if $w_{-c}\in q^{\N}$, that is, $c\in \tau(q^{\Z})$. 

The case $c\leq -2$ is treated similarly (and essentially contains the previous argument as a special case). Now it suffices to verify the existence of a $c$-set $Z\subseteq \Lambda_{x^2}$. We see by some elementary computation that Case \eqref{It1} arises if and only if $c_0<c$. Case \eqref{Itb} appears if and only if $w_c \in x^2q^{\Z}$, that is $c\in \tau(-x^2q^{\Z})$, and this set of $c$ is not enlarged in Case \eqref{Ita}. Note that this set contains in particular the boundary point $c_0$. Finally, Case \eqref{It2} only appears when $x^2\in q^{\Z}$, but in this case the value $c=-q-q^{-1}$ is contain in the set corresponding to Case \eqref{Itb}. 
\end{proof}

\begin{Rem} It is easy to see that $\UnitC{y}{z}\mathscr{A}_x\UnitC{y}{z} = \UnitC{y}{z}\mathrm{Pol}(\Omega)$, with $\mathrm{Pol}(\Omega)$ the polynomial algebra in $\Omega$. Hence $\UnitC{y}{z}A_x\UnitC{y}{z} = C_0(X_{y,z})$ for some compact subset $X_{y,z}\subseteq \R$, which can be described (with some more effort and in a more tedious way) along the lines of Corollary \ref{CorSpecUni}. In particular, the invariant state $\phi_{y,z}$ on $A_x$ corresponds to a probability measure on $X_{y,z}$. These probability measures can be shown to be Askey-Wilson measures (with respect to parameters determined in terms of $y$ and $z$), a result which is in essence already contained in \cite{KoR1}. In particular, it follows from this that the union of the supports of all these measures is the set $\lbrack -2,2\rbrack \cup \tau(q^{\Z})\cup \tau (-x^2q^{\Z})$, which is strictly smaller than $\Spec(\Omega)$. Consequently, the invariant weight $\phi$ on $A_x$ is not faithful. 
\end{Rem}

\end{document}